\documentclass[reqno,12pt]{amsart}

\usepackage[utf8]{inputenc}
\usepackage{microtype}
\usepackage{amsfonts}
\usepackage{amsmath}
\usepackage{graphicx}
\usepackage{hyperref}
\usepackage{color}

\usepackage{amssymb}
\usepackage{enumitem}

\usepackage{tikz}
\usetikzlibrary{topaths}

\usepackage{a4wide}

\usepackage{empheq}

\newtheorem{theorem}{Theorem}
\newtheorem*{theorem*}{Theorem}

\newtheorem{lemma}[theorem]{Lemma}

\newtheorem{proposition}[theorem]{Proposition}
\newtheorem*{proposition*}{Proposition}

\newtheorem{corollary}[theorem]{Corollary}
\newtheorem*{corollary*}{Corollary}
\newtheorem{conjecture}[theorem]{Conjecture}

\newtheorem*{conjecture*}{Conjecture}

\newtheorem*{question*}{Question}

\newtheorem*{main:main_Finfty}{Theorem~\ref{thrm:main_Finfty}}
\newtheorem*{main:main_separating}{Theorem~\ref{thrm:main_separating}}
\newtheorem*{main:main_embedding}{Theorem~\ref{thrm:embed_fin_gen}}

\theoremstyle{definition}
\newtheorem{definition}[theorem]{Definition}
\newtheorem{remark}[theorem]{Remark}
\newtheorem{example}[theorem]{Example}
\newtheorem{observation}[theorem]{Observation}

\usepackage{thmtools}
\usepackage{thm-restate} 

\newcommand{\N}{\mathbb{N}}

\newcommand{\R}{\mathbb{R}}

\newcommand{\G}{\mathcal{G}}   
\newcommand{\C}{\mathfrak{C}}  
\newcommand{\brick}{B}  
\newcommand{\canon}{h}  

\newcommand{\emptysequence}{\varnothing}
\newcommand{\V}{\mathcal{V}}

\newcommand{\lk}{\operatorname{lk}}

\newcommand{\dlk}{\lk^\downarrow\!}

\newcommand{\defeq}{\mathbin{\vcentcolon =}}
\newcommand{\newword}[1]{\textit{#1}}

\newcommand{\id}{\mathrm{id}}
\newcommand{\gtwist}{\mathrm{gtwist}}

\newcommand{\symgroup}[1]{\Sigma_{#1}}  
\newcommand{\SV}{SV}  
\newcommand{\SVG}{SV_G}  
\newcommand{\SVV}{S\mathcal{V}}  
\newcommand{\SVVG}{S\mathcal{V}_G}  
\newcommand{\twist}{\tau}

\usepackage[normalem]{ulem}

\DeclareMathOperator{\F}{F}

\DeclareMathOperator{\Stab}{Stab}

\DeclareMathOperator{\res}{res}
\DeclareMathOperator{\core}{core}
\DeclareMathOperator{\Spec}{Spec}
\DeclareMathOperator{\SSpec}{SSpec}

\numberwithin{equation}{section}

\begin{document}

\title{Twisted Brin--Thompson groups}
\date{\today}


\author[J.~Belk]{James Belk}
\address{School of Mathematics and Statistics, University of St Andrews, St Andrews, KY16~9SS, Scotland.} 
\email{jmb42@st-andrews.ac.uk}

\author[M.~C.~B.~Zaremsky]{Matthew C.~B.~Zaremsky}
\address{Department of Mathematics and Statistics, University at Albany (SUNY), Albany, NY 12222.}
\email{mzaremsky@albany.edu}

\begin{abstract}
We construct a family of infinite simple groups that we call \emph{twisted Brin--Thompson groups}, generalizing Brin's higher-dimensional Thompson groups~$sV$ ($s\in\N$). We use twisted Brin--Thompson groups to prove a variety of results regarding simple groups. For example, we prove that every finitely generated group embeds quasi-isometrically as a subgroup of a two-generated simple group, strengthening a result of Bridson. We also produce examples of simple groups that contain every~$sV$ and hence every right-angled Artin group, including examples of type~$\F_\infty$ and a family of examples of type $\F_{n-1}$ but not of type~$\F_n$, for arbitrary~$n\in\N$. This provides the second known infinite family of simple groups distinguished by their finiteness properties.
\end{abstract}

\maketitle
\thispagestyle{empty}

\renewcommand*{\thetheorem}{\Alph{theorem}} 
\section*{Introduction}

The Brin--Thompson groups $sV$ ($s\in\N$) are a family of groups introduced by Brin in \cite{brin04} as higher-dimensional generalizations of the classical Thompson group~$V=1V$. Each $sV$ acts by homeomorphisms on the Cantor space $\C^s$, where $\C=\{0,1\}^\omega$ is the usual Cantor set. Brin proved that all of these groups are finitely generated and simple~\cite{brin10}, and that $2V$ is finitely presented~\cite{brin05}.  Hennig and Matucci later showed that all of the groups $sV$ are finitely presented~\cite{hennig12}, and Bleak and Lanoue proved that $sV$ is not isomorphic to $s'V$ for $s\ne s'$~\cite{bleak10}.

In this paper we introduce a new family of infinite simple groups that we call \textit{twisted Brin--Thompson groups}.  Given any group $G$ acting faithfully on a countable set~$S$, the associated twisted Brin--Thompson group $\SVG$ acts by homeomorphisms on the Cantor space~$\C^S$.  The Brin--Thompson groups themselves correspond to the case where $S$ is finite and $G$ is trivial.

When $S$ is infinite, the group $\SVG$ is ``large'' in the sense that it contains $sV$ for all~$s\in \N$.  The first author, Bleak, and Matucci have proven that every finitely generated right-angled Artin group embeds into some $sV$~\cite{belk20}, and therefore $\SVG$ contains every finitely generated right-angled Artin group when $S$ is infinite.  It follows moreover from \cite{belk20} that such $\SVG$ contain all virtually special groups, for example all finitely generated Coxeter groups and many word hyperbolic groups.  Despite this ``largeness'', twisted Brin--Thompson groups can have surprisingly good finiteness properties.  For example, in Section~\ref{sec:gens} we prove:

\begin{restatable}{theorem}{finitegeneration}\label{thm:FinitelyGeneratedTheorem}The group $\SVG$ is finitely generated if and only if $G$ is finitely generated and the action of $G$ on $S$ has finitely many orbits.
\end{restatable}

In particular, if the action of $G$ on $S$ has $m$ orbits and $G$ has $n$ generators, then we prove that $\SVG$ is generated by two copies of $G$ and at most $m+n-1$ copies of the Brin--Thompson group~$2V$.  The groups $\SVG$ are also vigorous in the sense of Bleak, Elliott, and Hyde~\cite{hyde17,BEH20}, so it follows that $\SVG$ can be generated by two elements whenever it is finitely generated.  Indeed, a related generalization of the Brin--Thompson groups has been investigated in this context by Hyde~\cite{hyde17}.

There is a natural monomorphism $\iota_\varnothing\colon G\to \SV_G$, corresponding to the canonical action of $G$ on the Cantor space~$\C^S$. For $G$ and $\SVG$ finitely generated, we prove in Section~\ref{sec:arb_fin_lengths} that this embedding is geometrically well-behaved:

\begin{restatable}{theorem}{quasiretracts}\label{thm:quasi-retracts}Suppose $G$ and $\SVG$ are finitely generated. Then there exists a coarse Lipschitz map $\rho\colon \SVG\to G$ such that $\rho\circ \iota_\varnothing$ is the identity.  In particular, the group $\SVG$ quasi-retracts onto~$G$, and $\iota_\varnothing$ is a quasi-isometric embedding of $G$ into~$\SVG$.
\end{restatable}

See Section~\ref{sec:arb_fin_lengths} for the relevant definitions.  Since every finitely generated group $G$ acts faithfully and transitively on \textit{some} countable set $S$ (e.g., $S=G$), and since $\SVG$ is simple (Theorem~\ref{thrm:simple}), Theorems~\ref{thm:FinitelyGeneratedTheorem} and~\ref{thm:quasi-retracts} immediately yield the following:

\begin{theorem}Every finitely generated group embeds quasi-isometrically as a subgroup of a finitely generated (indeed, two-generated) simple group.\qed
\end{theorem}

Hall proved that every finitely generated group embeds into a finitely generated simple group~\cite{hall74}. Schupp~\cite{schupp76} and independently Goryushkin~\cite{goryushkin74} showed that the simple group can be taken to be two-generated, with Schupp's generators having orders two and~three.  However, it was not previously known that such embeddings could be made quasi-isometric.  By the result of~\cite{BEH20}, the two generators for $\SVG$ can be taken to have finite order, with one having order two, but we do not know whether the second generator can be made to have order three. The above result also strengthens a theorem of Bridson, who proved that every finitely generated group can be quasi-isometrically embedded as a subgroup of a finitely generated group that has no proper finite index subgroups~\cite{bridson98}.

As a remark, in our proof of Theorem~\ref{thm:quasi-retracts}, for our particular choices of finite generating sets, the quasi-isometric embedding is actually an isometric embedding. The fact that every finitely generated group embeds isometrically as a subgroup of a finitely generated simple group has also been proved independently by Darbinyan and Steenbock in \cite{darbinyan20}.

We can also prove some further finiteness properties for $\SVG$ in the case where the action of $G$ on $S$ is sufficiently transitive. Specifically, recall that a group $G$ of permutations of a set $S$ is \newword{oligomorphic} if for every $k\ge 1$ there are only finitely many orbits of $k$-element subsets of~$S$ \cite{cameron90}.  The proof of the following theorem occupies Sections~\ref{sec:cpxes}, \ref{sec:action}, and~\ref{sec:dlks}:

\begin{restatable}{theorem}{finitenesstheorem}\label{thm:FinitenessTheorem}Let $G$ be an oligomorphic group of permutations of a countable set~$S$. Let $n\in\mathbb{N}\cup\{\infty\}$, and suppose that
\begin{enumerate}
\item $G$ is of type\/~$\F_n$, and\smallskip
\item The stabilizer in $G$ of every finite subset of $S$ is of type\/~$\F_n$.
\end{enumerate}
Then $\SVG$ is of type\/~$\F_n$ as well.
\end{restatable}

Here we say that a group is of \newword{type\/~$\F_n$} if it acts geometrically (i.e., properly and cocompactly) on some $(n-1)$-connected CW~complex.  (A more standard definition is that a group is of type~$\F_n$ if it admits a classifying space with finite $n$-skeleton, but this is equivalent, for example thanks to Brown's Criterion \cite{brown87}.)  For example, a group is of type~$\F_1$ if and only if it is finitely generated, and a group is of type~$\F_2$ if and only if it is finitely presented.  A group is of \newword{type~$\F_\infty$} if it is of type~$\F_n$ for all $n\in\N$.

Brown and Geoghegan proved that Thompson's group $F$ is of type $\F_\infty$~\cite{BG84}, and Brown proved the same result for Thompson's groups $T$ and $V$~\cite{brown87}. Kochloukova, Mart\'{i}nez-P\'{e}rez, and Nucinkis proved that the Brin--Thompson groups $2V$ and $3V$ are of type $\F_\infty$~\cite{kochloukova13}, and Fluch, Marschler, Witzel, and the second author extended this result to all of the Brin--Thompson groups~\cite{fluch13}.  Our proof of the above theorem generalizes their techniques, and involves constructing a similar complex and analyzing the descending links, though the fact that $S$ may be infinite complicates the situation considerably.

We can use the above theorem to produce many new simple groups of type $\F_\infty$. The following corollary gives one example, which to the best of our knowledge is the first example of a finitely presented simple group that contains every right-angled Artin group.

\begin{corollary}\label{cor:ThompsonGroupExample}If $G$ is Thompson's group $F$ and $S$ is the set of dyadic rationals in~$(0,1)$, then $\SVG$ is a simple group of type\/~$\F_\infty$ that contains $sV$ for all~$s\in\N$, and hence contains every right-angled Artin group.
\end{corollary}
\begin{proof}Thompson's group $F$ is of type~$\F_\infty$ \cite{BG84}, and it acts transitively on $k$-element sets of dyadic rationals in $(0,1)$ for every~$k$ \cite[Lemma~4.2]{cannon96}.  The stabilizer in $F$ of any set of $k$ dyadic rationals is isomorphic to a direct product of $k+1$ copies of~$F$, and therefore such stabilizers are of type $\F_\infty$ as well. The result now follows from Theorem~\ref{thm:FinitenessTheorem}. 
\end{proof}

Note also that $\SVG$ is of type $\F_\infty$ whenever the set $S$ is finite.  When $G$ is non-trivial, the methods of Bleak and Lanoue from~\cite{bleak10} can be used to show that such a group does not belong to the Brin--Thompson family (since the corresponding groups of germs are non-abelian), but it would be interesting to classify such groups up to isomorphism.

Alonso proved that any quasi-retract of a group of type $\F_n$ must be of type~$\F_n$~\cite{alonso94}.  By Theorem~\ref{thm:quasi-retracts}, it follows that $G$ must be of type $\F_n$ whenever $\SVG$ is.  Combining this with Theorem~\ref{thm:FinitenessTheorem}, we obtain the following:

\begin{theorem}\label{thm:FinitenessLength}Let $G$ be an oligomorphic group of permutations of a set $S$.  Let $n\in\N$, and suppose that
\begin{enumerate}
\item $G$ is of type\/ $\F_{n-1}$ but not of type\/ $\F_n$, and\smallskip
\item The stabilizer in $G$ of any finite subset of $S$ is of type\/ $\F_{n-1}$.
\end{enumerate}
Then $\SVG$ is of type\/ $\F_{n-1}$ but not of type\/ $\F_n$. \qed
\end{theorem}

We can use this theorem to produce new simple groups that are of type $\F_{n-1}$ but not of type~$\F_n$.

\begin{corollary}\label{cor:HoughtonGroupExample}If $G$ is the $n$th Houghton group $H_n$ and $S=\{1,\ldots,n\}\times \N$, then $\SVG$ is a simple group of type\/~$\F_{n-1}$ that is not of type\/~$\F_n$.
\end{corollary}
\begin{proof}Brown proved that $H_n$ is of type $\F_{n-1}$ but not of type $\F_n$~\cite{brown87}.  It is easy to see that $H_n$ acts transitively on $k$-element subsets of $S$ for every~$k$, and the stabilizer of any $k$-element subset is isomorphic to $H_n\times \Sigma_k$, where $\Sigma_k$ denotes the $k$th symmetric group.  In particular, such stabilizers are all of type~$\F_{n-1}$.  The result now follows from Theorem~\ref{thm:FinitenessLength}.
\end{proof}

The family of groups in the corollary above is the second known family $\{G_n\}$ of simple groups such that each $G_n$ is of type\/ $\mathrm{F}_{n-1}$ but not\/~$\mathrm{F}_n$, the first being the family arising from R\"over--Nekrashevych groups described by Skipper, Witzel, and the second author in~\cite{skipper19}.  This also gives the third known infinite family of pairwise non-quasi-isometric simple groups, with the first such example found by Caprace and R\'emy \cite{caprace09,caprace10}, and the second being the Skipper--Witzel--Zaremsky family.

Finally, we should mention that Theorem~\ref{thm:FinitenessTheorem} does not appear to be sharp. We conjecture the following precise characterization of the finiteness properties of $\SVG$, which is true for $n=1$ as seen in Theorem~\ref{thm:FinitelyGeneratedTheorem}, and is reminiscent of the analogous characterization for the wreath products considered in~\cite{bartholdi15}.

\begin{conjecture}
The group $\SVG$ is of type\/ $\F_n$ if and only if the following conditions hold:
\begin{enumerate}
\item The action of $G$ on $S$ has finitely many orbits of $n$-element subsets. \smallskip
\item $G$ is of type\/ $\F_n$. \smallskip
\item For each $1\leq k< n$, the stabilizer in $G$ of any $k$-element subset of $S$ is of type\/~$\F_{n-k}$.
\end{enumerate} 
\end{conjecture}

For example, we conjecture that if $G$ is finitely presented, the action of $G$ on $S$ has finitely many orbits of pairs, and the stabilizer in~$G$ of each point of $S$ is finitely generated, then $\SVG$ is finitely presented.

\subsection*{Acknowledgments}  We are grateful to Collin Bleak, James Hyde, Francesco Matucci, Rachel Skipper, and Stefan Witzel for helpful conversations, and to the organizers of the 2018 Spring Topology \& Dynamical Systems Conference in Auburn, Alabama, where this project began. We also thank the anonymous referee for a number of helpful remarks. During the creation of this paper, the first author was partially supported by EPSRC grant EP/R032866/1 and the second author was partially supported by grant \#635763 from the Simons Foundation.

\numberwithin{theorem}{section}
\section{The groups}\label{sec:groups}

In this section we define the family of twisted Brin--Thompson groups.  We begin by briefly reviewing the definition of Brin's groups~$sV$, which we generalize slightly to a family of groups~$SV$, where $S$ is a set of any cardinality. In the case where $S$ is a finite set with $s$ elements, our group $SV$ is isomorphic to Brin's group~$sV$.

\subsection*{The groups \texorpdfstring{$\boldsymbol{SV}$}{SV}}
Let $\C$ be the Cantor set $\{0,1\}^\omega$. Given a set~$S$, the associated \newword{Cantor cube} $\C^S$ is the product space~$\prod_{s\in S}\C$, i.e., the set of all functions $S\to\C$.  

Let $\{0,1\}^*$ be the set of all finite binary sequences, including the empty sequence~$\emptysequence$.  We say that a function $\psi\colon S\to\{0,1\}^*$ has \newword{finite support} if $\psi(s) =\emptysequence$ for all but finitely many~$s\in S$.  Given such a function, the associated \newword{dyadic brick} in $S$ is the set
\[
\brick(\psi) = \{\kappa\in \C^S \mid \text{$\psi(s)$ is a prefix of $\kappa(s)$ for each $s\in S$}\}.
\]
Note that there is a \newword{canonical homeomorphism} $\canon_\psi\colon \C^S\to \brick(\psi)$ defined by
\[
\canon_\psi(\kappa)(s) = \psi(s)\cdot\kappa(s)
\]
for each $\kappa\in\C^S$ and $s\in S$, where $\cdot$ denotes concatenation.  More generally, if $\brick(\varphi)$ and $\brick(\psi)$ are dyadic bricks, we refer to the composition $\canon_{\psi}\circ \canon_{\varphi}^{-1}$ as the \newword{canonical homeomorphism} from $\brick(\varphi)$ to~$\brick(\psi)$.

Given any two partitions $\brick(\varphi_1),\ldots,\brick(\varphi_n)$ and $\brick(\psi_1),\ldots,\brick(\psi_n)$ of $\C^S$ into the same number of dyadic bricks, we can define a homeomorphism $\C^S\to \C^S$ by mapping each $\brick(\varphi_i)$ to the corresponding $\brick(\psi_i)$ by a canonical homeomorphism.  The group of all homeomorphisms of $\C^S$ defined in this fashion is the \newword{Brin--Thompson group}~$\SV$.

\begin{remark}\label{rem:DyadicPartitions}
If $\mathcal{P}$ is a partition of $\C^S$ into dyadic bricks, a \textit{very elementary expansion} of~$\mathcal{P}$ is a partition $\mathcal{P}'$ with the property that each dyadic brick in $\mathcal{P}$ is the union of at most two dyadic bricks in~$\mathcal{P}'$.  A \textit{dyadic partition} of $\C^S$ is any partition that can be obtained from the trivial one by a sequence of very elementary expansions.

The difference between arbitrary partitions into dyadic bricks and dyadic partitions has been a source of some confusion in the literature on Brin--Thompson groups.  Every partition of $\C^S$ into dyadic bricks is a dyadic partition as long as $|S|\leq 2$, but for $|S|\geq 3$ non-dyadic partitions exist, e.g., the partition
\[
B(0,0,0),\; B(1,1,1),\; B(0,1,\emptysequence),\; B(\emptysequence,0,1),\; B(1,\emptysequence,0)
\]
when $|S|=3$.  However, if $\mathcal{P}$ is any partition of $\C^S$ into dyadic bricks, there always exists a dyadic partition $\mathcal{P}'$ of $\C^S$ which can be obtained from $\mathcal{P}$ by a sequence of very elementary expansions.

If $g$ is an element of $\SV$, then it is always possible to find domain and range partitions for $g$ that are dyadic.  Something stronger is true: if $\mathcal{P}_1$ and $\mathcal{P}_2$ are any domain and range partitions for~$g$, then there exist domain and range dyadic partitions $\mathcal{P}_1'$ and $\mathcal{P}_2'$ for $g$ so that each $\mathcal{P}_i'$ can be obtained from $\mathcal{P}_i$ by a sequence of very elementary expansions.

In particular, the confusion in the literature does not result in any actual problems.  One can define $SV$ either using partitions into dyadic bricks or using dyadic partitions, and obtain the same group.
\end{remark}

\subsection*{The twisted group \texorpdfstring{$\boldsymbol{\SVG}$}{SV\_G}}
Now let $G$ be any group of permutations of the set $S$.  For each $\gamma\in G$, let $\twist_\gamma\colon \C^S\to \C^S$ be the \textit{twist homeomorphism} that permutes the coordinates of points in $\C^S$ according to~$\gamma$.  That is, $\twist_\gamma$ is the homeomorphism of $\C^S$ defined by
\[
\twist_\gamma(\kappa)(s) = \kappa(\gamma^{-1} s)
\]
for all $\kappa\in\C^S$ and $s\in S$. Note that this definition is the natural one since $\twist_\gamma(\kappa)(\gamma s)=\kappa(s)$ for all~$s\in S$, i.e.,~the $\gamma s$ coordinate of $\twist_\gamma(\kappa)$ is the same as the $s$ coordinate of~$\kappa$.  Note also that $\twist_{\gamma\gamma'}=\twist_\gamma \twist_{\gamma'}$ for all $\gamma,\gamma'\in G$.

More generally, given any two dyadic bricks $\brick(\varphi)$ and $\brick(\psi)$ and an element $\gamma\in G$, the associated \newword{twist homeomorphism} $\brick(\varphi)\to \brick(\psi)$ is the composition $\canon_\psi\circ \twist_\gamma \circ \canon_{\varphi}^{-1}$, where $\canon_\varphi\colon \C^S\to \brick(\varphi)$ and $\canon_\psi\colon \C^S\to \brick(\psi)$ are the canonical homeomorphisms.

\begin{example}
If $S$ is the two-element set $\{1,2\}$, then $\C^S = \C\times \C$ is known as the \newword{Cantor square}.  In this case, the twist homeomorphism $\twist_{(1\;2)}$ associated with the transposition $(1\;2)$ is the ``diagonal flip'' of the Cantor square, which switches the first and second coordinates of each point.   The transposition $(1\;2)$ also induces a twist homeomorphism between any two dyadic bricks in the Cantor square, which consists of a diagonal flip followed by translation and scaling.
\end{example}

Now, if $\brick(\varphi_1),\ldots,\brick(\varphi_n)$ and $\brick(\psi_1),\ldots,\brick(\psi_n)$ are any two partitions of $\C^S$ into the same number of dyadic bricks and $\gamma_1,\ldots,\gamma_n\in G$, we can define a homeomorphism of~$\C^S$ by mapping each $\brick(\varphi_i)$ to $\brick(\psi_i)$ via the twist homeomorphism associated to~$\gamma_i$.  The group of all homeomorphisms of $\C^S$ defined in this fashion is the \newword{twisted Brin--Thompson group}~$\SVG$.
For example, Figure~\ref{fig:element_of_2V_{S_2}} shows an element of $\SVG$, where $S=\{1,2\}$ and $G$ is cyclic of order two.
\begin{figure}[tb]
\includegraphics{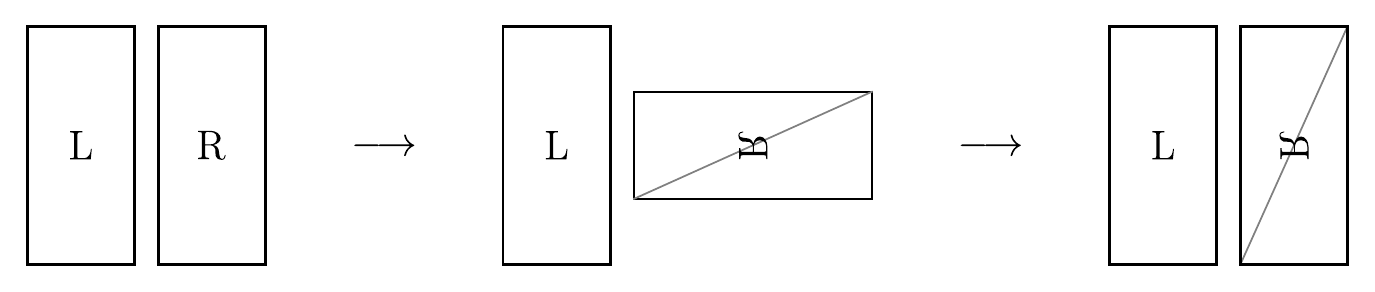}
\caption{An example of an element of $\SVG$, where $S=\{1,2\}$ and $G$ is cyclic of order two.}
\label{fig:element_of_2V_{S_2}}
\end{figure}

\section{The groupoids}\label{sec:groupoids}

In this section we define certain groupoids associated to twisted Brin--Thompson groups. These groupoids provide a convenient conceptual and notational framework for working with elements of the groups.

\subsection*{The spaces \texorpdfstring{$\boldsymbol{\C^S(n)}$}{C\textasciicircum s(n)}}

As before let $\C$ be the Cantor space $\{0,1\}^\omega$ and let $\C^S$ be the space of functions $S\to \C$. Let $\C^S_1,\C^S_2,\ldots$ be an infinite sequence of disjoint copies of $\C^S$, and for $m\in\N$ let
\[
\C^S(m) = \C^S_1 \cup \cdots \cup \C^S_m
\]
be a union of $m$ copies of $\C^S$.
A \newword{dyadic brick} in $\C^S(m)$ is any dyadic brick in any of the individual cubes $\C^S_1,\ldots,\C^S_m$.  Note that it makes sense to talk about canonical homeomorphisms and twist homeomorphisms between dyadic bricks in different cubes.

In the future, we will identify the standard Cantor cube $\C^S$ with the first cube $\C^S_1$, or equivalently with the space~$\C^S(1)$.

\subsection*{The groupoids \texorpdfstring{$\boldsymbol{\SVV}$ and $\boldsymbol{\SVVG}$}{SV and SV\_G}}
If $\brick_1,\ldots,\brick_k$ and $\brick_1',\ldots,\brick_k'$ are partitions of $\C^S(m)$ and $\C^S(n)$ respectively into the same number of dyadic bricks, we can define a homeomorphism $\C^S(m)\to \C^S(n)$ by mapping each $\brick_i$ to $\brick_i'$ by a canonical homeomorphism. The set of all such homeomorphisms is the \newword{Brin--Thompson groupoid~$\SVV$}.  If we instead map each $\brick_i$ to $\brick_i'$ via the twist homeomorphism associated with some $\gamma_i\in G$, then the set of all such homeomorphisms is the \newword{twisted Brin--Thompson groupoid~$\SVVG$}.

Note that the Brin--Thompson group $\SV$ is precisely the subgroup of $\SVV$ consisting of all homeomorphisms $\C^S(1)\to\C^S(1)$.  Similarly, the twisted Brin--Thompson group $\SVG$ is the subgroup of $\SVVG$ consisting of all homeomorphisms $\C^S(1)\to\C^S(1)$.

\subsection*{Germinal twists}
If $h\in\SVG$ and $\kappa$ is a point in $\C^S$, then there exists a dyadic brick~$B$ containing $\kappa$, a dyadic brick $B'$ containing $h(\kappa)$, and an element $\gamma\in G$ so that $h$ maps $B$ to $B'$ via the twist homeomorphism determined by $\gamma$.  In this case, we say that $\gamma$ is the \newword{germinal twist} for $h$ at~$\kappa$, denoted $\gtwist_\kappa(h)$.  Note that this does not depend on the bricks $B$ and $B'$ chosen, for if $B''$ is a sub-brick of $B$, then $h(B'')$ must be a sub-brick of $B'$, and $h$ maps $B''$ to $h(B'')$ by the twist homeomorphism associated to the same element $\gamma\in G$.

The notion of a germinal twist extends to the groupoid $\SVVG$ in an obvious way.  Note that a homeomorphism $h\colon \C^S(m)\to\C^S(n)$ in $\SVVG$ lies in $\SVV$ if and only if the germinal twist at each point in $\C^S(m)$ is the identity element of~$G$.

Note that the germinal twist is multiplicative, in the sense that
\[
\gtwist_\kappa(h_2h_1)= \gtwist_{h_1(\kappa)}(h_2)\,\gtwist_\kappa(h_1)
\]
for any $h_1,h_2\in \SVG$ and $\kappa\in\C^S$.  For $h_1,h_2\in \SVVG$, it follows that
\begin{enumerate}
\item $\SV\hspace{0.08333em}h_1=\SV\hspace{0.08333em}h_2$ if and only if $\gtwist_\kappa(h_1)=\gtwist_\kappa(h_2)$ for all $\kappa\in\C^S$, and\smallskip
\item $\SV\hspace{0.08333em}h_1\hspace{0.08333em}\SV=\SV\hspace{0.08333em}h_2\hspace{0.08333em}\SV$ if and only if
\[
\{\gtwist_\kappa(h_1) \mid \kappa\in\C^S\} = \{\gtwist_\kappa(h_2) \mid \kappa\in\C^S\}.
\]
\end{enumerate}

\subsection*{Direct sum}
In addition to their groupoid structure, $\SVV$ and $\SVVG$ have an additional operation $\oplus$ that we refer to as the \newword{direct sum}.   Given any two elements $h\colon \C^S(m)\to \C^S(n)$ and $h'\colon \C^S(m')\to \C^S(n')$ of $\SVG$, their direct sum is the homeomorphism
\[
h\oplus h'\colon \C^S(m+m')\to \C^S(n+n')
\]
that maps the first $m$ cubes of $\C_{m+m'}^S$ to the first $n$ cubes of $\C_{n+n'}^S$ via $h$, and maps the $m'$ remaining cubes of $\C_{m+m'}^S$ to the $n'$ remaining cubes of $\C_{n+n'}^S$ via~$h'$.

\subsection*{Permutations}
Given a permutation $\sigma$ of $\{1,\ldots,n\}$, the associated \newword{permutation homeomorphism}
\[
p_\sigma\colon \C^S(n)\to \C^S(n)
\]
in $\SVV$ is the homeomorphism that permutes the $n$ Cantor cubes of $\C^S(n)$ via $\sigma$.  That is, $p_\sigma$ is the homeomorphism of $\C^S(n)$ mapping each Cantor cube $\C^S_i$ to $\C^S_{\sigma(i)}$ via the canonical homeomorphism.

\subsection*{\texorpdfstring{$\boldsymbol{G}$-twists}{G-twists} and twisted permutations}
Given $\gamma_1,\ldots,\gamma_n\in G$ with associated twist homeomorphisms $\twist_{\gamma_1},\ldots,\twist_{\gamma_n}$ of $\C^S$, the direct sum
\[
t= \twist_{\gamma_1}\oplus \cdots \oplus  \twist_{\gamma_n}
\]
will be referred to as a \newword{$G$-twist} of~$\C^S(n)$, or simply a \newword{twist} if the group $G$ is clear.  The group of all $G$-twists of $\C^S(n)$ is isomorphic to the direct product $G^n=G\times \cdots \times G$.

If $t$ is a $G$-twist of $\C^S(n)$ and $p_{\sigma}$ is a permutation homeomorphism of $\C^S(n)$, the composition $tp_\sigma$ will be called a \newword{twisted permutation} of $\C^S(n)$.  Let $\G(n)$ be the group of all twisted permutations of $\C^S(n)$; this is isomorphic to the wreath product $G\wr \symgroup{n} = G^n\rtimes \symgroup{n}$, where $\symgroup{n}$ is the symmetric group of rank~$n$.

\subsection*{Simple splits}
Given an $s\in S$, the \textit{simple split with color $s$} is the homeomorphism
\[
x_s\colon \C^S(1) \to \C^S(2)
\]
in $\SVV$ defined as follows.  Let $\{B_{s,0},B_{s,1}\}$ be the partition of $\C^S(1) = \C^S$ into two bricks cut along the $s$ coordinate, so $B_{s,i}$ is the dyadic brick consisting of all $\kappa\in\C^S$ such that $\kappa(s)=i$.  Then the simple split $x_s$ is the homeomorphism that maps $B_{s,0}$ to the first cube $\C_1^S$ of $\C^S(2)$ and $B_{s,1}$ to the second cube $\C_2^S$ of $\C^S(2)$ by canonical homeomorphisms.

The following lemma records some basic relations in the groupoid $\SVVG$, all of which are easy to check.

\begin{lemma}[Relations in $\SVVG$]\label{lem:Relations}\quad
\begin{enumerate}
\item[\normalfont (1)] If $\sigma,\sigma'\in \symgroup{n}$ then $p_{\sigma}p_{\sigma'}=p_{\sigma\sigma'}$.\smallskip
\item[\normalfont (2)] If $\sigma\in \symgroup{n}$ and $\sigma'\in\symgroup{n'}$, then $p_{\sigma}\oplus p_{\sigma'}=p_{\sigma''}$ for some $\sigma''\in \symgroup{n+n'}$.\smallskip
\item[\normalfont (3)] If $\gamma,\gamma'\in G$ then $\twist_{\gamma} \twist_{\gamma'}=\twist_{\gamma\gamma'}$.\smallskip
\item[\normalfont (4)] If $h,h',h''\in \SVVG$ then $(h\oplus h')\oplus h'' = h\oplus (h'\oplus h'')$.\smallskip
\item[\normalfont (5)] If $h_1,h_1',h_2,h_2'\in \SVVG$ and $h_1 h_1'$ and $h_2 h_2'$ are both defined, then $(h_1\oplus h_2)(h_1'\oplus h_2') = (h_1 h_1')\oplus (h_2 h_2')$.\smallskip
\item[\normalfont (6)] If $h_1,\ldots,h_n\in \SVG$ and $\sigma\in\symgroup{n}$ then $(h_1\oplus\cdots\oplus h_n)p_\sigma = p_\sigma(h_{\sigma(1)}\oplus\cdots\oplus h_{\sigma(n)})$.\smallskip
\item[\normalfont (7)] If $s\in S$ and $\gamma\in G$ then $x_{\gamma s}\twist_\gamma = (\twist_\gamma\oplus \twist_\gamma)x_s$.\smallskip
\item[\normalfont (8)] If $s,t\in S$ and $s\ne t$ then $(x_t\oplus x_t) x_s = p_{(2\;3)}(x_s\oplus x_s)x_t$, where $(2\;3)\in\symgroup{4}$.\qed
\end{enumerate}
\end{lemma}

\subsection*{Multicolored trees}
A homeomorphism $f\colon \C^S\to \C^S(n)$ is called a \newword{multicolored tree} if $f$ maps the bricks of some dyadic partition $B_1,\ldots,B_n$ of $\C^S$ (see Remark~\ref{rem:DyadicPartitions}) to the Cantor cubes $\C^S_1,\ldots,\C^S_n$ of $\C^S(n)$ by canonical homeomorphisms. 

Equivalently, $f$ is a multicolored tree if either $f$ is the identity on $\C^S$ or
\[
f=p_\sigma(f_0\oplus f_1)x_s
\]
for some permutation homeomorphism $p_\sigma$ of $\C^S(n)$, some multicolored trees $f_0,f_1$, and some simple split $x_s$.  Note that a given multicolored tree might have more than one description of this form, as in Lemma~\ref{lem:Relations}(8).  

\subsection*{Multicolored forests}
A \newword{multicolored forest} is any homeomorphism $\C^S(m)\to \C^S(n)$ of the form
\[
p_\sigma(f_1\oplus \cdots \oplus f_m)
\]
where $p_\sigma$ is a permutation homeomorphism of $\C^S(n)$ and $f_1,\ldots,f_m$ are multicolored trees.  Note that the set of all multicolored forests is closed under composition and direct sums.  Indeed, the set of multicolored forests is precisely the smallest set of elements of $\SVVG$ which contains all permutations and simple splits and is closed under composition and direct sum.

We can also describe multicolored forests in terms of dyadic partitions.  We say that a partition $B_1,\ldots,B_n$ of $\C^S(m)$ into dyadic bricks is a \newword{dyadic partition} if its restriction to each $\C_i^S$ is a dyadic partition of~$\C_i^S$.  Then $f\colon \C^S(m)\to\C^S(n)$ is a multicolored forest if and only if it maps the bricks of some dyadic partition $B_1,\ldots,B_n$ of~$\C^S(m)$ to the Cantor cubes $\C_1^S,\ldots,\C_n^S$ of $\C^S(n)$ by canonical homeomorphisms.  This gives a one-to-one correspondence between multicolored forests $\C^S(m)\to \C^S(n)$ and ordered dyadic partitions $(B_1,\ldots,B_n)$ of~$\C^S(m)$.

The following lemma describes the structure of arbitrary elements of $\SVG$.

\begin{lemma}\label{lem:ForestTwistForest}Every element of $\SVVG$ can be written as $f_2^{-1}tf_1$ for some multicolored forests $f_1,f_2$ and some $G$-twist~$t$.  Such an element lies in $\SVV$ if and only if $t$ is the identity.
\end{lemma}
\begin{proof}Given a homeomorphism $h\colon \C^S(m)\to \C^S(n)$ in $\SVVG$, let $B_1,\ldots,B_k$ be a domain partition and $B_1',\ldots,B_k'$ a range partition so that $h$ maps each $B_i$ to $B_i'$ via the twist homeomorphism associated to some~$\gamma_i\in G$.  Refining if necessary as in Remark~\ref{rem:DyadicPartitions}, we may assume that $B_1,\ldots,B_k$ and $B_1',\ldots,B_k'$ are dyadic partitions.  Then
\[
h = f_2^{-1}(\twist_{\gamma_1}\oplus\cdots\oplus \twist_{\gamma_k})f_1
\]
where $f_1\colon \C^S(m)\to \C^S(k)$ is the multicolored forest that maps each $B_i$ to $\C^S_i$ by a canonical homeomorphism, and $f_2\colon \C^S(n)\to\C^S(k)$ is the multicolored forest that maps each $B_i'$ to~$\C^S_i$ by a canonical homeomorphism.  Moreover, since $h$ has germinal twist~$\gamma_i$ on each~$B_i$, it lies in $\SVV$ if and only if $\gamma_1,\ldots,\gamma_k$ are all the identity, which occurs if and only if $\twist_{\gamma_1}\oplus\cdots\oplus \twist_{\gamma_k}$ is the identity homeomorphism of~$\C^S(k)$.
\end{proof}

It follows from this lemma that every element of the group $\SVG$ can be written as $f_2^{-1}tf_1$ for some multicolored trees $f_1,f_2$ and some $G$-twist~$t$.

It will be very important in what follows to understand how multicolored forests interact with twisted permutations. The following lemma establishes the relationship.

\begin{lemma}\label{lem:swap}
If $f$ is a multicolored forest and $g$ is a twisted permutation such that the composition $fg$ is defined, then $fg=g'f'$ for some twisted permutation $g'$ and some multicolored forest $f'$. Moreover, if $f$ is a multicolored tree then $g'$ is a direct sum of copies of $g$.
\end{lemma}
\begin{proof}By the relations in parts (6) and (7) of Lemma~\ref{lem:Relations}, this holds whenever $f$ is a permutation homeomorphism or a simple split.  But if the statement holds for multicolored forests $f_1$ and $f_2$, then it clearly holds for $f_1f_2$ if this is defined, and moreover it holds for $f_1\oplus f_2$ by relation~(5) in  Lemma~\ref{lem:Relations}.  Since every multicolored forest can be obtained from permutation homeomorphisms and simple splits by taking iterated compositions and direct sums, the result follows.
\end{proof}

See Figure~\ref{fig:swap} for an example of Lemma~\ref{lem:swap}.

\begin{figure}[tb]
\includegraphics{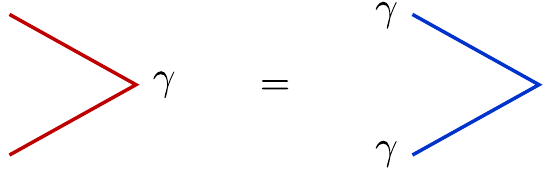}
\caption{With $S=\{r,b\}$ and $\gamma$ satisfying $r=\gamma b$, a picture of the equality $x_r \twist_\gamma = (\twist_\gamma\oplus \twist_\gamma)x_{b}$.}
\label{fig:swap}
\end{figure}

\begin{remark}\label{rmk:cloning}
For readers familiar with \cite{witzel18,witzel19}, the formula $fg=g'f'$ provides a Zappa--Sz\'ep structure, which could be developed into a notion of ``multicolored cloning systems'' using multicolored forests. A number of aspects of the general framework break down when there are infinitely many colors though, so it is unlikely this generalization would be helpful for our current setup. In the future, developing the frameworks in \cite{witzel18,witzel19} to allow for multicolored forests could be interesting.
\end{remark}

\subsection*{Subsets of \texorpdfstring{$\boldsymbol{S}$}{S} and spectrum}
Each subset $T$ of $S$ corresponds to a subgroup $TV^{(S)}$ of $SV$ which is isomorphic to $TV$.  Specifically, if we decompose $\C^S$ as a product $\C^T\times \C^{S\setminus T}$, then $TV^{(S)}$ is the subgroup of elements of $SV$ that act as an element of $TV$ on the $\C^T$ factor and act as the identity on the $\C^{S\setminus T}$ factor.  Similarly, the groupoid $S\V$ has a natural subgroupoid $T\V^{(S)}$ which is isomorphic to~$T\V$.

If $h\in S\V$, the \newword{spectrum} of $h$ is the finite set $\Spec(h)\subseteq S$ of all coordinates that $h$ uses in an essential way.  That is, $\Spec(h)$ is the smallest set $T\subseteq S$ such that $h\in T\V^{(S)}$. This definition has the following properties:
\begin{enumerate}
\item $\Spec(h)$ is empty if and only if $h$ is a permutation homeomorphism.\smallskip
\item $\Spec(x_s) = \{s\}$ for any simple split $x_s$.\smallskip
\item $\Spec(h_1\oplus h_2)=\Spec(h_1)\cup \Spec(h_2)$ for any $h_1,h_2\in\SVV$.\smallskip
\item $\Spec(h_1h_2)\subseteq \Spec(h_1)\cup \Spec(h_2)$ for any $h_1,h_2\in\SVV$ such that $h_1h_2$ is defined, with equality if $h_1$ and $h_2$ are multicolored forests.
\end{enumerate}
Note that these properties can be used to inductively compute the spectrum of any multicolored forest.  For example,
\[
\Spec\bigl(p_{(2\;3\;5)}((x_s\oplus\id_1)x_t\oplus x_u)x_v\bigr) = \{s,t,u,v\}
\]
for any $s,t,u,v\in S$, where $\id_1$ denote the identity map on~$\C^S$.

\section{Generators and simplicity}\label{sec:gens}

We will need two embeddings of $G$ into $\SVG$.  First, let $\iota_\emptysequence\colon G\to \SVG$ be the embedding
\[
\iota_\varnothing(\gamma) = \twist_\gamma \text{.}
\]
Note that this is a homomorphism by part (3) of Lemma~\ref{lem:Relations}. Next, fix an $s\in S$, and define an embedding $\iota_1^s \colon G \to \SVG$ by
\[
\iota_1^s(\gamma) = x_{s}^{-1}({\id_1} \oplus \twist_\gamma)x_{s}
\]
where $\id_1$ denotes the identity map on $\C^S$.  Thus $\iota_1^s(\gamma)$ applies the twist $\twist_\gamma$ to the ``right half'' of $\C^S$ determined by the direction~$s$.

\begin{proposition}\label{prop:BasicGenSets}
$\SVG$ is generated by $\SV\cup \iota_1^s(G)$.
\end{proposition}

\begin{proof}Let $H = \langle \SV\cup \iota_1^s(G)\rangle$. Since the double coset $\SV g\hspace{0.08333em}\SV$ of an element $g\in \SVG$ is determined by its set of germinal twists, it suffices to prove that there exists an element of $H$ with any given set $\{\gamma_1,\ldots,\gamma_n\}$ of germinal twists.

Note first that elements of $\SV$ have germimal twist set $\{1\}$, and elements of $\iota_1^s(G)$ have germinal twist set $\{1,\gamma\}$ for some $\gamma\in G$, so $H$ contains all elements having these sets of germinal twists.  Next note that if $g\in H$ has germinal twist set $\{\gamma_1,\ldots,\gamma_n\}$ and $\gamma\in G$, then there exist $g_1,g_2\in H$  with germinal twist set $\{1,\gamma\}$ so that $g_1g$ has germinal twist set $\{\gamma\gamma_1\}\cup \{\gamma_2,\ldots,\gamma_n\}$ and $g_2g$ has germinal twist set $\{\gamma\gamma_1\}\cup \{\gamma_1,\ldots,\gamma_n\}$.  It follows easily by induction that $H$ contains elements with any given set of germinal twists.
\end{proof}	

To obtain finite generation conditions for $\SVG$, we need some information about generating sets for~$SV$.  For $T\subseteq S$, recall that $SV$ has a natural subgroup $TV^{(S)}$ that is isomorphic to $TV$.  For the following proposition, we are particularly interested in the case where $T$ is a two-element subset of~$S$, so $TV^{(S)}\cong 2V$.

\begin{proposition}\label{prop:GeneratorsSV}Let $S$ be a set with $|S|\geq 2$, let\/ $\Gamma$ be a connected simple graph with $V(\Gamma) = S$, and for each edge $e\in E(\Gamma)$ let $eV^{(S)}$ be the associated copy of $2V$.  Then the union $\bigcup_{e\in E(\Gamma)} eV^{(S)}$
generates $SV$.
\end{proposition}
\begin{proof}In \cite[Theorem~25]{hennig12}, Hennig and Matucci describe a set of generators for the group $nV = \{1,\ldots,n\}V$, each of which is contained in $\{1,d\}V^{(\{1,\ldots,n\})}$ for some $2 \leq d \leq n$.  Thus the given proposition holds when $S=\{1,\ldots,n\}$ and $E(\Gamma) = \bigl\{\{1,2\},\{1,3\},\ldots,\{1,n\}\bigr\}$, i.e.,~whenever the graph $\Gamma$ is a finite star.

For $S$ finite, we can now prove the proposition by induction on~$|S|$.  The case where $|S|= 2$ is done, so suppose that $|S| \geq 3$.  Let $u$ be a vertex of $\Gamma$ whose removal does not disconnect~$\Gamma$, and let $\Gamma'$ be the (connected) subgraph of $\Gamma$ induced by $S\setminus u$.  Let $v$ be a vertex of $\Gamma$ that is adjacent to~$u$.  By our induction hypothesis, the copies of $2V$ corresponding to the edges of~$\Gamma'$ generate $(S\setminus u)V^{(S)}$, so we can replace $\Gamma'$ by the star $\bigl\{\{v,w\}\bigr\}_{w\in S\setminus\{u,v\}}$.  Since $\{u,v\}\in E(\Gamma)$, the graph $\Gamma$ now contains the entire star $\bigl\{\{v,w\}\bigr\}_{w\in S\setminus v}$, so by induction the proposition holds when $S$ is finite.

The case where $S$ is infinite follows easily, since any infinite connected graph is the union of its finite connected subgraphs.
\end{proof}

\begin{corollary}\label{cor:gen_sets}
Let\/ $\Gamma$ be a connected graph with vertex set $S$ that is invariant under the action of~$G$, and let\/ $\{e_\alpha\}_{\alpha\in \mathcal{I}}$ be a system of representatives for the orbits of $E(\Gamma)$ under~$G$.  Then $\SVG$ is generated by the subgroups $e_\alpha V^{(S)}$ (each of which is isomorphic to~$2V$) together with $\iota_\emptysequence(G)$ and~$\iota_1^s(G)$.
\end{corollary}
\begin{proof}For $\gamma\in G$, conjugation by $\iota_\emptysequence(\gamma)$ maps $e_\alpha V^{(S)}$ to $(\gamma e_\alpha)V^{(S)}$.  So the subgroup generated by the $\iota_\emptysequence(G) \cup \bigcup_{\alpha\in\mathcal{I}}e_\alpha V^{(S)}$ contains $eV^{(S)}$ for every $e\in E(\Gamma)$, and thus contains $SV$ by Proposition~\ref{prop:GeneratorsSV}.  By Proposition~\ref{prop:BasicGenSets}, it follows that the given generators generate~$\SVG$.
\end{proof}

\finitegeneration*
\begin{proof} First suppose $G$ is finitely generated and the action of $G$ on $S$ has finitely many orbits. Let $s_1,\ldots,s_m$ be representatives for the orbits of $G$ in~$S$, let $\gamma_1,\ldots,\gamma_n$ be generators for $G$, and let $\Gamma$ be the graph with vertex set $S$ and edge set
\[
\bigl\{ \{\gamma s_1, \gamma s_i\} \;\bigr|\; 2\leq i\leq m, \gamma\in G\bigr\} \cup \bigl\{ \{\gamma s_1, \gamma \gamma_j s_1 \} \;\bigr|\; 1\leq j \leq n\text{, }\gamma_js_1\ne s_1\text{, and }\gamma\in G\bigr\}.
\]
Then $\Gamma$ is connected and $G$-invariant, and has finitely many orbits of edges under~$G$.  By Corollary~\ref{cor:gen_sets}, it follows that $\SVG$ is generated by the subgroups $\{s_1,s_i\} V^{(S)}$ ($2\leq i\leq m$) and $\{s_1,\gamma_j s_1\}V^{(S)}$ ($1\leq j\leq n$ and $\gamma_js_1\ne s_1$) together with $\iota_\emptysequence(G)$ and $\iota_1^s(G)$, all of which have finitely many generators.

Now suppose $G$ is not finitely generated, so there is a filtration $G_1<G_2<G_3<\cdots$ of $G$ by proper subgroups. Then $SV_{G_1}\leq SV_{G_2}\leq SV_{G_3} \leq \cdots$ is a chain of subgroups whose union is~$\SVG$.  Moreover, for each $i$ the set of all germinal twists of all elements of $SV_{G_i}$ equals $G_i$, so each $SV_{G_i}$ is a proper subgroup of $SV_{G_{i+1}}$. We conclude that $\SVG$ is not finitely generated.

Lastly, suppose the action of $G$ on $S$ has infinitely many orbits. Then there is a filtration $S_1\subset S_2\subset S_3\subset\cdots$ of $S$ by proper $G$-invariant subsets. Fix an $s\in S_1$, and for each~$i$ let $H_i$ be the subgroup of $\SVG$ generated by $S_iV^{(S)}$ and~$\iota_1^s(G)$.  (This subgroup is isomorphic to $S_iV_G$ if $G$ acts faithfully on~$S_i$, but in general $S_iV_G$ is a quotient of~$H_i$.) Since each $S_i$ is $G$-invariant, we know that $H_i \cap SV = S_iV^{(S)}$, and therefore each~$H_i$ is a proper subgroup of~$H_{i+1}$. By Proposition~\ref{prop:BasicGenSets} the union of the $H_i$'s is all of~$\SVG$, and therefore $\SVG$ is not finitely generated.
\end{proof}

In the case where $G$ is Thompson's group $F$ acting on the dyadics $S$ in $(0,1)$, as in Corollary~\ref{cor:ThompsonGroupExample}, the proof of the theorem above tells us that $\SVG$ is generated by $\{1/2,3/4\}V^{(S)}$ (a copy of~$2V$) together with $\iota_\emptysequence(F)$ and $\iota_{1}^{1/2}(F)$ (two copies of~$F$).  If instead $G$ is Houghton's group $H_n$ and $S=\{1,\ldots,n\}\times\mathbb{N}$ as in Corollary~\ref{cor:HoughtonGroupExample}, a similar argument shows that $\SVG$ is generated by one copy of $2V$ and two copies of~$H_n$. (In general, $\SVG$ is generated by one copy of~$2V$ and two copies of~$G$ whenever the action of $G$ on $S$ is primitive, since in this case $G$ has a generating set with only one generator that moves~$s_1$.)

\begin{theorem}\label{thrm:simple}
The twisted Brin--Thompson group $\SVG$ is simple. Moreover, if $\SVG$ is finitely generated then it can be generated by two elements of finite order.
\end{theorem}
\begin{proof}In \cite{hyde17,BEH20}, Bleak, Elliott, and Hyde prove that a group $K$ of homeomorphisms of a Cantor space $X$ is simple provided that it satisfies the following three conditions:
\begin{enumerate}
\item $K$ is \textit{vigorous}.  That is, for every clopen set $A\subseteq X$ and proper clopen subsets $B$ and $C$ of $A$ there exists a $k\in K$ supported on $A$ such that $k(B)\subseteq C$.\smallskip
\item $K$ is generated by elements of \textit{small support}, i.e., those elements which are supported on proper clopen subsets of~$X$.\smallskip
\item $K$ is \textit{perfect}, i.e., $[K,K]=K$.
\end{enumerate}
Moreover, they prove that every simple, vigorous, finitely generated group of homeomorphisms of the Cantor set is generated by two elements of finite order. Thus it suffices to verify that $\SVG$ satisfies the three conditions above.

Clearly $SV$ is vigorous, so $SV_G$ is as well.  Moreover, since $SV$ is generated by elements of small support and each $\iota_1^s(\gamma)$ has small support, it follows from Proposition~\ref{prop:BasicGenSets} that $SV_G$ is generated by elements of small support.  To prove that $SV_G$ is perfect, recall first from \cite{brin10} that $SV$ is perfect, and therefore $[SV_G,SV_G]$ contains~$SV$.  By Proposition~\ref{prop:BasicGenSets}, it suffices to prove that $[SV_G,SV_G]$ contains $\iota_1^s(\gamma)$ for some fixed $s\in S$ and all $\gamma\in G$.  Let $B=x_s^{-1}(\C_2^S)$ be the support of~$\iota_1^s(\gamma)$, and let $f$ be any element of $SV$ that maps $B$ to a proper subset of itself.  Then the commutator $k=[\iota_1^s(\gamma),f] = \iota_1^s(\gamma)f\iota_1^s(\gamma)^{-1} f^{-1}$ has germinal twist $\gamma$ on $B\setminus f(B)$ and the identity everywhere else.  Since $\iota_1^s(\gamma)$ also has germinal twists $\gamma$ and the identity, there exist elements $h_1,h_2\in SV$ so that $h_1 k h_2 = \iota_1^s(\gamma)$, and therefore $\iota_s^1(\gamma)\in [SV_G,SV_G]$.
\end{proof}

\section{Quasi-retracts}\label{sec:arb_fin_lengths}

If $X$ and $Y$ are metric spaces, a function $f\colon X\to Y$ is said to be \newword{coarse Lipschitz} if there exist constants $C,D>0$ so that
\[
d(f(x),f(x'))\leq C\,d(x,x')+D
\]
for all $x,x'\in X$.  For example, any homomorphism between finitely generated groups is coarse Lipschitz with respect to the word metrics.

A function $\iota\colon X\to Y$ is said to be a \newword{quasi-isometric embedding} if there exist constants $C,D>0$ so that
\[
\frac{1}{C} \,d(x,x') - D \leq d(f(x),f(x'))\leq C\,d(x,x')+D
\]
for all $x,x'\in X$.  Clearly any quasi-isometric embedding is coarse Lipschitz, but the converse need not hold.

A function $\rho\colon X\to Y$ is said to be a \newword{quasi-retraction} if it is coarse Lipschitz and there exists a coarse Lipschitz function $\iota\colon Y\to X$ and a constant $E>0$ so that $d(\rho\circ \iota(y),y) \leq E$ for all $y\in Y$.  If such a function $\rho$ exists, then $Y$ is said to be a \newword{quasi-retract} of~$X$.  In this case, it is not hard to show that $\iota$ must be a quasi-isometric embedding of $Y$ into~$X$.

Now suppose $G$ and $\SVG$ are finitely generated, with word metrics corresponding to some finite generating sets. Our next result is that $G$ is a quasi-retract of $\SVG$. In order to prove this we need a way of canonically extracting an element of $G$ from an element of $SV_G$. We will do this using the notion of germinal twists $\gtwist_\kappa(h)$ from Section~\ref{sec:groupoids}.

\quasiretracts*
\begin{proof}Fix a point $\kappa\in \C^S$, and let $\rho\colon \SVG\to G$ be the function
\[
\rho(h) = \gtwist_\kappa(h).
\]
Note then that $\rho\circ \iota_\varnothing$ is the identity function on $G$.  We claim that $\rho$ is coarse Lipschitz with respect to the \textit{left} word metrics on $\SVG$ and~$G$.

Fix an $s\in S$ and a generating set $A_G$ for $G$, and recall from Corollary~\ref{cor:gen_sets} that $\SVG$ is generated by $\iota_\varnothing(A_G)\cup \iota_1^s(A_G)$ together with finitely many elements of~$SV$.  But:
\begin{enumerate}
\item $\rho\bigl(\iota_\varnothing(a)\hspace{0.08333em}h\bigr) = a\hspace{0.08333em}\rho(h)$ for all $a\in A_G$ and $h\in \SVG$, \smallskip
\item $\rho\bigl(\iota_1^s(a)\hspace{0.08333em}h\bigr) \in \{\rho(h),a\hspace{0.08333em}\rho(h)\}$ for all $a\in A_G$ and $h\in\SVG$, and\smallskip
\item $\rho(h'h)=\rho(h)$ for all $h'\in \SV$ and $h\in \SVG$.
\end{enumerate}
It follows that $\rho$ is nonexpanding and hence coarse Lipschitz.  Since $\iota_\varnothing$ is a homomorphism, it must be coarse Lipschitz as well.  We conclude that $\rho$ is a quasi-retraction and $\iota_\varnothing$ is a quasi-isometric embedding. (In fact for these choices of generating sets, $\iota_\varnothing$ is even an isometric embedding.)
\end{proof}

\section{Complexes}\label{sec:cpxes}

In this section we construct complexes on which the twisted Brin--Thompson groups $\SVG$ act. First we define a directed poset $P_1$ on which $\SVG$ acts and then build the so called Stein complex $X$. Most of the ideas here are directly inspired by the construction of the analogous complexes for $sV$ in \cite{fluch13}.  However, unlike the complex in \cite{fluch13}, the Stein complex here will be locally infinite if $S$ is infinite, which makes a number of arguments more complicated.

\subsection*{The poset \texorpdfstring{$\boldsymbol{P}$}{P}}

If $h\colon \C^S(m')\to \C^S(m)$ is an element of $\SVVG$ write $[h]$ for the coset $\mathcal{G}(m)\hspace{0.08em}h$, where $\G(m)$ is the group of twisted permutations of~$\C^S(m)$. Let $P$ be the set of equivalence classes $[h]$ of elements $h$ of $\SVVG$.

If $h\colon \C^S(m')\to \C^S(m)$ is an element of $\SVVG$, we will refer to $m$ as the \newword{rank} of $h$ and $m'$ as the \newword{corank} of~$h$.  Note that rank and corank are invariant under multiplication by twisted permutations, so it makes sense to refer to the \emph{rank} and \emph{corank} of an element of $P$.  

We can place a partial order on $P$ as follows.  Given $v,w\in P$, we say that $w$ is an \newword{expansion} of $v$, denoted $v\leq w$, if there exists an $h\in\SVVG$ and a multicolored forest $f$ so that $v=[h]$ and~$w=[fh]$.

\begin{lemma}
The relation $\le$ is a partial order on $P$.
\end{lemma}
\begin{proof}
Clearly $\leq$ is reflexive.  For antisymmetry, suppose that $v\leq w$ and $w\leq v$ for some $v,w\in P$.  Then there exist $h_1,h_2\in\SVVG$ and multicolored forests $f_1,f_2$ so that $v=[h_1]=[f_2h_2]$ and $w=[h_2]=[f_1h_1]$.  Then $h_1=g_2f_2h_2$ and $h_2=g_1f_1h_1$ for some twisted permutations $g_1$ and $g_2$.  It follows that $g_1f_1g_2f_2$ is the identity, so $f_1$ and $f_2$ must be permutations.  Then $g_1f_1$ is a twisted permutation, so $[h_1]=[h_2]$ and hence $v=w$.

For transitivity suppose that $u\leq v$ and $v\leq w$ for some $u,v,w\in P$.  Then there exist $h_1,h_2\in\SVVG$ and multicolored forests $f_1,f_2$ so that $u=[h_1]$, $v=[h_2]=[f_1h_1]$, and $w=[f_2h_2]$.  Let $g$ be a twisted permutation so that $h_2=gf_1h_1$, and observe that $f_2g=g'f_2'$ for some twisted permutation $g'$ and multicolored forest $f_2'$ by Lemma~\ref{lem:swap}.  Then $w=[f_2h_2]=[f_2gf_1h_1]=[g'f_2'f_1h_1]=[f_2'f_1h_1]$, and since $f_2'f_1$ is a multicolored forest it follows that $u\leq w$.
\end{proof}

Let $P_1$ be the subposet of $P$ consisting of corank-$1$ elements. We are especially interested in the restriction of $\le$ to $P_1$. Recall that a poset is \newword{directed} if any two elements have a common upper bound. It is a standard fact that directed posets have contractible geometric realizations.

\begin{proposition}\label{prop:directed}
The poset $P_1$ is directed, so its geometric realization $|P_1|$ is contractible.
\end{proposition}

\begin{proof}Let $[h_1],[h_2]\in P_1$.  By Lemma~\ref{lem:ForestTwistForest}, there exist multicolored forests $f_1,f_2$ and a $G$-twist $t$ so that $h_1h_2^{-1} = f_2^{-1}tf_1$, and rearranging gives $f_2 h_1 = tf_1h_2$. Then $[h_1]\leq [f_2h_1]$ and $[h_2]\leq [f_1h_2] = [tf_1h_2] = [f_2h_1]$, so $[f_2h_1]$ is a common upper bound.
\end{proof}

\subsection*{Elementary bricks and elementary multicolored forests}
Our next goal is to define the Stein complex for $\SVG$, which is a certain contractible subcomplex of~$|P_1|$.  This will be based on the notion of elementary bricks and elementary multicolored forests.

We say that a dyadic brick $B$ in $\C^S$ is \newword{elementary} if its defining function $\psi\colon S\to \{0,1\}^*$ has the property that $\psi(s)\in \{0,1,\emptysequence\}$ for all $s\in S$.  Such a brick can be obtained from $\C^S$ by cutting at most once in each direction.

A dyadic partition $\{B_1,\ldots,B_n\}$ of $\C^S$ is called \newword{elementary} if each of the bricks $B_i$ is an elementary brick.  Similarly, a dyadic partition of $\C^S(n)$ is elementary if it restricts to an elementary partition of each of the Cantor cubes~$\C_1^S,\ldots,\C_n^S$.

A multicolored forest $f\colon \C^S(m)\to \C^S(n)$ is called \newword{elementary} if there exists an elementary dyadic partition $B_1,\ldots,B_n$ of $\C^S(m)$ such that $f$ maps each $B_i$ to $\C_i^S$ by a canonical homeomorphism.

\subsection*{The lattice of forests}

Let $F_m$ be the subset of $P$ consisting of all vertices of corank $m$ that can be represented by multicolored forests. A vertex of $F_m$ is called \newword{elementary} if it can be represented by an elementary multicolored forest.

\begin{proposition}\label{prop:JoinForests}Every pair of elements $[f_1],[f_2]\in F_m$ have a least upper bound $[f_1]\lor [f_2]$ in $F_m$.  If $[f_1]$ and $[f_2]$ are elementary then so is $[f_1]\lor [f_2]$.
\end{proposition}
\begin{proof}Recall that the multicolored forests $f\colon \C^S(m)\to \C^S(n)$ are in one-to-one correspondence with ordered dyadic partitions $(B_1,\ldots,B_n)$ of $\C^S(m)$, where $f$ maps each $B_i$ to $\C^S_i$ by a canonical homeomorphism.  If $f$ and $f'$ are multicolored forests corresponding to dyadic partitions $(B_1,\ldots,B_r)$ and $(B_1',\ldots,B_s')$ of $\C^S(m)$, then it is easy to check that:
\begin{enumerate}
\item $[f] = [f']$ if and only if $\{B_1,\ldots,B_r\}=\{B_1',\ldots,B_s'\}$, and\smallskip
\item $[f]\leq [f']$ if and only if $\{B_1',\ldots,B_s'\}$ is a refinement of $\{B_1,\ldots,B_r\}$.
\end{enumerate}
That is, the vertices in $F_m$ are in one-to-one correspondence with (unordered) dyadic partitions of $\C^S(m)$, with the partial order $\leq$ corresponding to refinement. But any two dyadic partitions $\{B_1,\ldots,B_r\}$ and $\{B_1',\ldots,B_s'\}$ of $\C^S(m)$ have a least common refinement, namely the dyadic partition consisting of all nonempty intersections $B_i\cap B_j'$.  Moreover, since the nonempty intersection of elementary bricks is an elementary brick, the least common refinement of two elementary dyadic partitions is again elementary.
\end{proof}

Recall that a partially ordered set is called a \newword{lattice} if every pair of elements $v,w$ have a least upper bound $v\lor w$, called the \newword{join}, and a greatest lower bound $v\land w$, called the \newword{meet}.

\begin{corollary}Each poset $F_m$ is a lattice, as is the subset of $F_m$ consisting of elementary vertices.
\end{corollary}
\begin{proof}Note that $F_m$ has a unique minimal element, namely the equivalence class of the trivial corank-$m$ forest.  Moreover, every element of $F_m$ has finitely many predecessors with respect to~$\leq$.  Since finite subsets of $F_m$ have least upper bounds by Proposition~\ref{prop:JoinForests}, it follows easily that any finite subset of $F_m$ has a greatest lower bound as well.  A similar argument applies to the elementary vertices.
\end{proof}

\subsection*{The Stein complex}
Given vertices $v,w\in P$, we say that $w$ is an \newword{elementary expansion} of $v$, denoted $v\preceq w$, if there exists an $h\in \SVVG$ and an elementary multicolored forest $f$ so that $v=[h]$ and $w=[fh]$. See Figure~\ref{fig:elementary} for an example.
\begin{figure}[tb]
\includegraphics{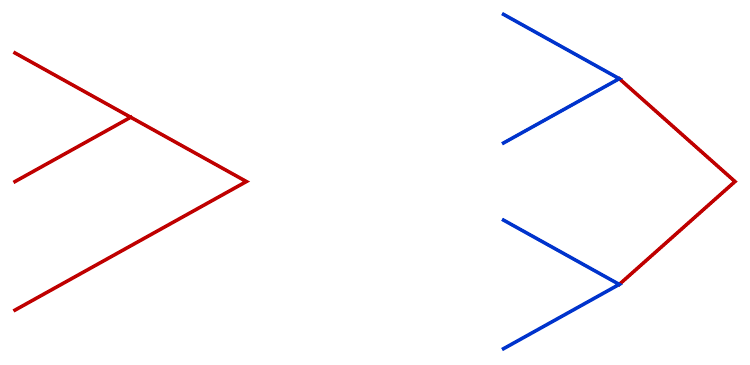}
\caption{Let $S=\{r,b\}$. The first picture is the multicolored forest $f_1=(x_r\oplus \id_1) x_r$, which is not elementary. The second picture is the multicolored forest $f_2=(x_b\oplus x_b) x_r$, which is elementary.}
\label{fig:elementary}
\end{figure}

Now, recall that a simplex of the geometric realization $|P_1|$ is a finite chain $v_0<\cdots < v_k$ of elements of~$P_1$.  Such a simplex is \newword{elementary} if $v_0\preceq v_k$ (and hence $v_i\preceq v_j$ for all $i\leq j$).  The elementary simplices form a subcomplex $X$ of $|P_1|$, which we refer to as the \newword{Stein complex} for~$\SVG$. The term ``Stein complex'' is a reference to the analogous complex constructed by Stein for Thompson's group $F$ in \cite{stein92}.

We will use interval notation in $P$, for instance $[v,w]=\{u\mid v\le u\le w\}$, with similar definitions for open and half-open intervals. Call an interval $[v,w]$ \emph{elementary} if $v\preceq w$.

\begin{proposition}[Elementary core]\label{prop:elem_core}
For any $v<w$ in $P$ the set of elementary expansions of $v$ that lie in $[v,w]$ has a maximum element, which is not equal to $v$.
\end{proposition}
\begin{proof}Without loss of generality $v=[\id_m]$ and $w=[f]$ for some multicolored forest $f$ of corank $m$.  Then the elementary expansions of $v$ are precisely the elementary multicolored forests of corank~$m$, of which only finitely many lie in $[v,w]$.  By Proposition~\ref{prop:JoinForests}, these elementary multicolored forests have a least upper bound~$u$.  This is again an elementary multicolored forest and it lies in $[v,w]$, so $u$ is the desired maximum.

Moreover, since $v\ne w$ the multicolored forest $f$ is not a permutation, so we can write
\[
f = p_\sigma(f_1\oplus f_2)x_s
\]
for some permutation homeomorphism $p_\sigma$, some multicolored forests $f_1$ and $f_2$, and some simple split $x_s$.  Then $[x_s]$ is an elementary expansion of $v$ that lies in $[v,w]$ and is strictly greater than $v$, so the maximum elementary expansion of $v$ in $[v,w]$ is not equal to $v$.
\end{proof}

Given an interval $[v,w]$ in $P$, we refer to the maximum elementary expansion of $v$ that it contains as the \newword{elementary core} of $w$ relative to~$v$, denoted $\core_v(w)$.

\begin{proposition}\label{prop:cible}
The Stein complex $X$ is homotopy equivalent to $|P_1|$, hence contractible.
\end{proposition}

\begin{proof}
We proceed by a now-standard strategy (see for example \cite{brown92,fluch13,bux16,belk16}). Observe that each interval $[v,w]$ in $P_1$ is itself a directed set, whose geometric realization $|[v,w]|$ is a contractible subcomplex of~$|P_1|$.  Since every finite chain in $P_1$ is contained in an interval, every simplex in $|P_1|$ is contained in some $|[v,w]|$, so $|P_1|$ is the union of these subcomplexes.  Note that the Stein complex $X\subseteq |P_1|$ is precisely the union of the subcomplexes $|[v,w]|$ for which $[v,w]$ is an elementary interval.

Now, define the \textit{length} of an interval $[v,w]$ to be $m'-m$, where $m$ is the rank of $v$ and $m'$ is the rank of~$w$.  Then we can write $|P_1|$ as an ascending union of subcomplexes
\[
X=X_0 \subseteq X_1 \subseteq X_2 \subseteq \cdots
\]
where each $X_n$ is the union of $X_{n-1}$ with the geometric realizations of all non-elementary intervals $[v,w]$ of length~$n$.  It suffices to prove that each inclusion $X_{n-1}\to X_n$ is a homotopy equivalence.  

Note first that each simplex in $X_n$ which is not in $X_{n-1}$ lies in the geometric realization of a unique non-elementary interval $[v,w]$ of length~$n$.  Thus it suffices to prove that the intersection $|[v,w]|\cap X_{n-1}$ is contractible for each non-elementary interval $[v,w]$ of length~$n$.  This intersection is precisely the union $|[v,w)|\cup |(v,w]|$, which is isomorphic to the suspension of $|(v,w)|$ with $v$ and $w$ serving as suspension points, so it suffices to prove that each $|(v,w)|$ is contractible.

Let $r\colon |(v,w)|\to |(v,\core_v(w)]|$ be the retraction induced by the order-preserving function $\core_v\colon (v,w)\to (v,\core_v(w)]$.  Quillen observed in \cite[Section~1.3]{quillen78} that if $f,g\colon A\to B$ are order-preserving functions between posets and $f(a)\leq g(a)$ for all $a\in A$ then the induced maps from $|A|$ to $|B|$ are homotopic.  Since $\core_v(x)\leq x$ for all $x\in (v,w)$, it follows that $r$ is homotopic to the identity map on $|(v,w)|$.  But $|(v,\core_v(w)]|$ is clearly contractible, being the cone of $|(v,\core_v(w))|$ with cone point~$\core_v(w)$, and therefore $|(v,w)|$ is contractible as well.
\end{proof}

\section{Actions on the complexes}\label{sec:action}

In this section we prove that if $G$ is oligomorphic and if every stabilizer in $G$ of a finite subset of $S$ is of type $\F_n$ (including $G=\Stab_G(\emptyset)$ itself), then the only remaining step toward proving that the twisted Brin--Thompson group $SV_G$ is of type $\F_n$ (Theorem~\ref{thm:FinitenessTheorem}) is a descending link analysis. This analysis will then happen in the next section. It seems worth remarking that these conditions on orbits and stabilizers of finite subsets are very similar to those used in \cite{bartholdi15} to characterize finiteness properties of wreath products.

The strategy is similar in spirit to that used in \cite{fluch13} for the Brin--Thompson groups~$sV$. For those familiar with the approach, the problem here is that elementary intervals can be arbitrarily long if we have an infinite set of colors $S$. This makes it difficult to apply any existing techniques that have been used for related groups. For instance one could try and use the ``cloning system'' tools from \cite{witzel18} and \cite{witzel19}, but these are quickly stymied by the fact that elementary intervals can be arbitrarily long, so the analog of the crucial Observation~4.6 in \cite{witzel18} fails (but see Remark~\ref{rmk:cloning} for further thoughts on the connection to cloning systems). For the same reason, the ``groundedness'' of Belk and Forrest (see \cite[Theorem~6.2]{belk16} and \cite[Theorem~4.9]{belk19}) does not apply directly.

We will use a combination of Bestvina--Brady Morse theory \cite{bestvina97} and Brown's Criterion \cite{brown87} that by now is quite standard, which we set up now.

\begin{definition}[Morse function, sublevel set, descending link]
Let $Y$ be an affine cell complex and $\phi \colon Y\to \R$ a map. We call $\phi$ a \emph{Morse function} if it is affine and non-constant on positive dimensional cells and the image $\phi(Y^{(0)})\subseteq \R$ is closed and discrete. For $m\in\R$, the \emph{sublevel complex} $Y_m$ is the induced subcomplex of $Y$ with vertex set $Y^{(0)}\cap \phi^{-1}((-\infty,m])$. For a vertex $y\in Y$ the \emph{descending link} $\dlk y$ of $y$ with respect to $\phi$ is the link of $y$ in $Y_{\phi(y)}$. 
\end{definition}

\begin{theorem}\label{thrm:morse_brown}
Let $H$ be a group acting cellularly on a contractible affine cell complex $Y$ and let $\phi\colon Y\to \R$ be an $H$-invariant Morse function. Suppose that each sublevel set $Y_m$ is $H$-cocompact, that the stabilizer in $H$ of any cell of $Y$ is of type\/ $\F_n$, and that for each $k\in\N$ there exists $m\in\N$ such that for all vertices $y\in Y$ with $\phi(y)\ge m$, the descending link\/ $\dlk y$ is $k$-connected. Then $H$ is of type\/~$\F_n$.
\end{theorem}

Note that $SV_G$ acts by right multiplication on the set of corank-$1$ elements of $S\mathcal{V}_G$, which is compatible with left multiplication by $\mathcal{G}$, and so induces an action of $SV_G$ on $P_1$. This action is compatible with $\le$ and so extends to an action on $|P_1|$, under which $X$ is invariant.

\begin{lemma}\label{lem:vtx_stab}
The stabilizer in $SV_G$ of a vertex of $X$ with rank $m$ is isomorphic to~$G\wr \symgroup{m}$.
\end{lemma}
\begin{proof}
We have that $a\in SV_G$ stabilizes $[h]$ if and only if $hah^{-1} \in \mathcal{G}$ (recall that $\mathcal{G}$ is the set of twisted permutations). If $h$ has rank $m$ this is equivalent to $hah^{-1} \in \mathcal{G}(m)$, so we have $\Stab_{SV_G}([h]) = h^{-1}(\mathcal{G}(m))h$. Since $\mathcal{G}(m)\cong G\wr \symgroup{m}$, we conclude that $\Stab_{SV_G}([h]) \cong G\wr \symgroup{m}$.
\end{proof}

\begin{lemma}\label{lem:stab_spectrum}
Let $\gamma\in G$ and let $f$ be a multicolored tree. If $\gamma$ fixes\/ $\Spec(f)$ pointwise then\/ $[f\twist_\gamma]=[f]$, and if\/ $[f\twist_\gamma]=[f]$ then $\gamma$ stabilizes\/ $\Spec(f)$.
\end{lemma}
\begin{proof}If $x_s$ is a simple split and $\gamma s=s$ then $x_s \twist_\gamma=(\twist_\gamma\oplus \twist_\gamma)x_s$ by Lemma~\ref{lem:Relations}(7).  It follows easily by induction that if $f$ is a multicolored tree and $\gamma$ fixes $\Spec(f)$ pointwise then $f\twist_\gamma = (\twist_\gamma\oplus \cdots \oplus \twist_\gamma)f$, and therefore $[f\twist_\gamma]=[f]$.

For the second statement, let $f$ be any multicolored tree and suppose that $[f\twist_\gamma]=[f]$.  By Lemma~\ref{lem:swap}, there exists a $G$-twist $g$ and a multicolored forest $f'$ so that $f\twist_\gamma = gf'$.  Note then that $\Spec(f) = \gamma\,\Spec(f')$.  But $[f']=[gf'] = [f\twist_\gamma]=[f]$ and hence $\Spec(f')=\Spec(f)$, so $\gamma$ must stabilize~$\Spec(f)$.
\end{proof}

For the following proposition, recall that a group $H_1$ is \newword{commensurate with} a group $H_2$ if there exist finite index subgroups of $H_1$ and $H_2$ that are isomorphic.

\begin{proposition}\label{prop:stabs}
Let\/ $\Delta$ be a simplex in\/ $|P_1|$ corresponding to a chain $v_0 < \cdots < v_k$, where $v_0=[h]$ and $v_k=[(f_1\oplus\cdots\oplus f_n)h]$ for some $h\in\SVVG$ and some multicolored trees $f_1,\ldots,f_n$.  Then the stabilizer of\/ $\Delta$ in $\SVG$ is commensurate with\/ $\prod_{i=1}^n \mathrm{Stab}_G(\Spec(f_i))$.
\end{proposition}
\begin{proof}Since the action of $\SVG$ is order-preserving, any element of $\SVG$ that stabilizes this simplex must fix $v_0$ and $v_k$.  Moreover, since the interval $[v_0,v_k]$ is finite, the stabilizer of $\Delta$ has finite index in the stabilizer of $[v_0,v_k]$.

Now, the stabilizer of $v_0$ in $\SVG$ is $h^{-1}\mathcal{G}(n)\hspace{0.0833em}h$, which is isomorphic to $G\wr \symgroup{n}$ by Lemma~\ref{lem:vtx_stab}.  This has a finite index subgroup isomorphic to $G^n$, consisting of all
$h^{-1}(\twist_{\gamma_1}\oplus\cdots\oplus \twist_{\gamma_n})h$ for $\gamma_1,\ldots,\gamma_n\in G$.  Such an element fixes $v_k$ if and only if $[f_i\twist_{\gamma_i}] = [f_i]$ for all~$i$.  By Lemma~\ref{lem:stab_spectrum}, the group of all such tuples $(\gamma_1,\ldots,\gamma_n)$ is commensurate with $\prod_{i=1}^n \mathrm{Stab}_G(\Spec(f_i))$.
\end{proof}

\begin{corollary}\label{cor:stabs}If $G$ is of type\/ $\F_n$ and the stabilizer in $G$ of any finite subset of $S$ is of type\/ $\F_n$, then the stabilizer of any simplex in\/ $|P_1|$ is of type\/~$\F_n$.
\end{corollary}
\begin{proof}
Finite products of groups of type $\F_n$ are of type $\F_n$, and any group commensurate to a group of type $\F_n$ is of type $\F_n$, so this is immediate from Proposition~\ref{prop:stabs}.
\end{proof}

Let $\phi \colon S\mathcal{V}_G \to \N$ be the function ``rank'' that sends $h\colon \C^S(m') \to \C^S(m)$ to $m$. Denote also by $\phi$ the restriction of $\phi$ to $P_1=X^{(0)}$, and the extension of this to a map $X\to\R$ given by affinely extending to simplices. Note that $\phi$ is $SV_G$-invariant. It is a Morse function since it is non-constant on edges, by virtue of a non-trivial multicolored forest having different rank and corank.

For each $m\in\N$ consider the sublevel set $X_m \defeq \phi^{-1}([1,m])$. Since $\phi$ is invariant under the action of $SV_G$, each $X_m$ is $SV_G$-invariant.

\begin{proposition}[Cocompact]\label{prop:cocpt}
If the action of $G$ on $S$ is oligomorphic then each $X_m$ is $SV_G$-cocompact.
\end{proposition}

\begin{proof}
First note that for any $m\in\N$, $SV_G$ acts transitively on the set of vertices $v$ with $\phi(v)=m$. This is because if $\phi([h])=\phi([h'])$ then the product $h^{-1}h'$ makes sense, lies in $SV_G$, and takes $[h]$ to $[h']$. In particular there are finitely many $SV_G$-orbits of vertices in $X_m$. It now suffices to show that for a vertex $v\in X_m$, say with $\phi(v)=r\le m$, there are finitely many $\Stab_{SV_G}(v)$-orbits of simplices in $X_m$ containing $v$ as their vertex of minimum rank. Without loss of generality $v=[\id_r]$, so we want to show that there are finitely many $\mathcal{G}(r)$-orbits of simplices in $X_m$ containing $[\id_r]$ as their vertex of minimum rank. For each $i\in\N$ choose a subset $T_i\subseteq S$ with $|T_i|=i$. Since $G$ is oligomorphic, the proof of Lemma~\ref{lem:swap} implies that there exists $i\in\N$ such that every such $\mathcal{G}(r)$-orbit contains a simplex of the form $[\id_r]<[f_1]<\cdots<[f_k]$ for multicolored forests $f_1,\dots,f_k$ such that $\Spec(f_1),\dots,\Spec(f_k)\subseteq T_i$. The result now follows from the fact that for each $i,m\in\N$ there are only finitely many multicolored forests with spectrum lying in $T_i$ and with rank at most $m$.
\end{proof}

At this point if $G$ is oligomorphic and if every stabilizer in $G$ of a finite subset of $S$ is of type $\F_n$, then the twisted Brin--Thompson group $SV_G$ acts cellularly on the contractible complex $X$, with stabilizers of type $\F_n$, and we have an $SV_G$-cocompact filtration $(X_m)_{m\in\N}$ for the Morse function $\phi$. In order to show that $SV_G$ is of type $\F_n$ it now suffices by Theorem~\ref{thrm:morse_brown} to prove that as $\phi([f])$ goes to infinity, the descending link $\dlk [f]$ of $[f]$ in $X$ becomes arbitrarily highly connected. This is the content of the next section.

\section{Descending links}\label{sec:dlks}

For $[h]\in P_1$ the descending link $\dlk [h]$ of $[h]$ in $X$ is the subcomplex of $X_{\phi([h])}$ supported on vertices of the form $[f^{-1}gh]$ for $g$ a twisted permutation and $f$ a non-trivial elementary multicolored forest. If $\phi([h])=m$ then $\dlk [h] \cong \dlk [\id_m]$. Let us denote this complex by $E_m$. A vertex of $E_m$ is an element of $P$ of the form $v=[f^{-1} g]$ for $f$ a non-trivial elementary multicolored forest and $g\in \mathcal{G}(m)$.

\subsection*{Nerve lemma}
In the course of analyzing $E_m$ we will need to use the Nerve Lemma, which we record here for reference. Recall that the \emph{nerve} of a covering of a complex $Z$ by subcomplexes $Z_1,\dots,Z_m$ is the simplicial complex with a vertex for each $Z_i$ and such that $Z_{i_1},\dots,Z_{i_k}$ span a simplex whenever $Z_{i_1}\cap\dots\cap Z_{i_k} \ne\emptyset$.

\begin{lemma}[Nerve Lemma]\cite[Lemma~1.2]{bjoerner94}\label{lem:nerve}
Let $Z$ be a simplicial complex covered by subcomplexes $Z_1,\dots,Z_m$. Suppose for each $1\le r\le m$ that every non-empty intersection of $r$ of these subcomplexes is $(n-r)$-connected. Then $Z$ is $(n-1)$-connected if and only if the nerve of the covering is $(n-1)$-connected.
\end{lemma}

\subsection*{The subcomplex \texorpdfstring{$\boldsymbol{VE_m}$}{VE\_m}}

We say that a multicolored forest $f$ is \newword{very elementary} if it can be written as a direct sum
\[
f = f_1\oplus \cdots \oplus f_m
\]
where each $f_i$ is either the identity homeomorphism on $\C^S$ or a simple split. Clearly any very elementary multicolored forest is elementary.

Let $VE_m$ be the subcomplex of $E_m$ spanned by vertices of the form $v=[f^{-1}g]$ for $f$ a non-trivial very elementary multicolored forest, and for each $k\in\N$ let
\[
\nu(k) = \left\lfloor\frac{k-2}{3}\right\rfloor.
\]

\begin{lemma}\label{lem:VE_conn}
The complex $VE_m$ is $(\nu(m)-1)$-connected.
\end{lemma}

\begin{proof}
We induct on $m$ (the base case that $VE_m$ is non-empty for $m\ge 2$ is trivially satisfied). Write $\C^S(m)=\C^S_1\sqcup\cdots\sqcup \C^S_m$. For $1\le i<j\le m$ say that a vertex $[f^{-1} g]$ of $VE_m$ \emph{matches} $i$ and $j$ if $f^{-1} g$ sends $\C^S_i$ and $\C^S_j$ into the same copy of $\C^S$ in its range. Note that this is independent of the choice of representative, since postcomposition by a twisted permutation just maps the copies of $\C^S$ in the range of $f^{-1} g$ homeomorphically to each other. For each $1\le i<j\le m$ let $VE_m^{i,j}$ be the subcomplex of $VE_m$ spanned by vertices $v$ that either match $i$ and $j$, or else do not match $i$ or $j$ to any $k$. By construction $VE_m$ is covered by the subcomplexes $VE_m^{i,j}$.

We will apply the Nerve Lemma~\ref{lem:nerve} to the covering of $VE_m$ by the $VE_m^{i,j}$. Note that $VE_m^{i,j}$ is the union of the (contractible) stars of its rank-$(m-1)$ vertices that match $i$ and $j$ (with one for each element of $S$). The intersection of two or more of these stars retracts onto the subcomplex spanned by those $v$ that do not match $i$ or $j$ to any $k$. This is isomorphic to $VE_{m-2}$, which is $(\nu(m)-2)$-connected by induction. This shows that each $VE_m^{i,j}$ is $(\nu(m)-1)$-connected. It is straightforward to see that $VE_m^{i_1,j_1},\dots,VE_m^{i_r,j_r}$ have non-empty intersection if and only if the sets $\{i_1,j_1\},\dots,\{i_r,j_r\}$ are pairwise disjoint. Any non-empty such $VE_m^{i_1,j_1}\cap\dots\cap VE_m^{i_r,j_r}$ is therefore the union of (contractible) stars of its rank-$(m-r)$ vertices that match $i_\ell$ to $j_\ell$ for each $1\le \ell\le r$ (with one for each element of $S^r$). The intersection of two or more of these stars retracts onto the subcomplex spanned by those $v$ that do not match any $i_\ell$ or $j_\ell$ to any $k$. This is isomorphic to $VE_{m-2r}$, which is $(\nu(m)-r-1)$-connected by induction. This shows that $VE_m^{i_1,j_1}\cap\dots\cap VE_m^{i_r,j_r}$ is $(\nu(m)-r)$-connected. In particular the hypotheses of the Nerve Lemma are all satisfied.

The nerve of this covering consists of a vertex for each $VE_m^{i,j}$ and a simplex spanned by any collection of these with non-empty intersection. As we said, $VE_m^{i_1,j_1}\cap\dots\cap VE_m^{i_r,j_r}$ is non-empty if and only if the sets $\{i_1,j_1\},\dots,\{i_r,j_r\}$ are pairwise disjoint. Hence the nerve is precisely the matching complex of the complete graph $K_m$, which is $(\nu(m)-1)$-connected (\cite[Theorem~4.1]{bjoerner94}).
\end{proof}

\subsection*{The Morse function \texorpdfstring{$\boldsymbol{\mu}$}{mu}}

Now we want to build up from $VE_m$ to $E_m$ and get high connectivity for $E_m$. We will use a Morse function similar to the one used in \cite{fluch13} for (what here would be called) the $|S|<\infty$ and $G=\{1\}$ case. Let $f$ be an elementary multicolored forest with rank $m$ and corank $q$. Write the domain of $f$ as
\[
\C^S(q)=R_1\sqcup\cdots\sqcup R_q
\]
and the range as
\[
\C^S(m)=L_1\sqcup\cdots\sqcup L_m \text{,}
\]
with each $R_i$ and $L_i$ a copy of $\C^S$, and call the $R_i$ \emph{roots} and the $L_i$ \emph{leaves}. The image of each root under $f$ is a union of some number of leaves. Define the \emph{weight} of a root to be the number of leaves in its image under $f$. For each $k\in \N$ define the \emph{$k$th weight multiplicity} $\mu_k(f)$ of $f$ to be the number of roots with weight $k$. Note that the $\mu_k$ are invariant under left and right multiplication by permutations and twists, so it makes sense to define the $k$th weight multiplicity of a vertex $v=[f^{-1} g]\in E_m$ to be
\[
\mu_k(v)\defeq \mu_k(f) \text{.}
\]
Let $\mu\colon E_m^{(0)} \to \N_0^m$ (where $\N_0\defeq \N\cup\{0\}$) be the function
\[
\mu\defeq (\mu_m,\dots,\mu_3,\phi) \text{.}
\]
Note that $\mu_2$ and $\mu_1$ are not included in this tuple. Consider the lexicographic ordering on~$\N_0^m$, and choose any order-preserving embedding $\N_0^m \to \R$, to view $\mu$ as a map from $E_m^{(0)}$ to $\R$. Extend this affinely to the simplices of $E_m$ to get a map $\mu\colon E_m \to \R$.

\begin{observation}
The map $\mu\colon E_m \to \R$ is a Morse function.
\end{observation}

\begin{proof}
Since $\mu_k \le m$ for all $k$ and $\phi$ restricted to $E_m$ is bounded by $m-1$, the image of $E_m^{(0)}$ under $\mu$ is finite, hence closed and discrete. By construction, $\mu$ restricts to an affine function on each simplex. On any edge of $E_m$, $\phi$ is non-constant, hence $\mu$ is non-constant.
\end{proof}

As a remark, the Morse function $\mu$ is defined slightly differently than the corresponding Morse function in \cite{fluch13}, though it dictates the same behavior for descending links.

\subsection*{The \texorpdfstring{$\boldsymbol{\mu}$}{mu}-descending link}
Note that $VE_m$ is precisely the sublevel set $E_m^{\mu\le (0,\dots,0,m-1)}$ since a vertex $v\in E_m$ lies in $VE_m$ if and only if $\mu_k(v)=0$ for all $k\ge3$ (this is exactly why we did not use $\mu_2$ or $\mu_1$ in the definition of $\mu$). In particular, to inspect the higher connectivity of $E_m$ it suffices to analyze $\mu$-descending links $\dlk_\mu v$ for vertices $v\in E_m$ satisfying $\mu(v)>(0,\dots,0,m-1)$, i.e., satisfying $\mu_k(v)>0$ for some $k\ge 3$. Every vertex of $\dlk_\mu v$ is either an upper bound or a lower bound of $v$ (as elements of $P$), and any upper bound of $v$ is an upper bound of any lower bound of $v$, so $\dlk_\mu v$ decomposes as the join of two subcomplexes, the \emph{$\mu$-descending split link} and \emph{$\mu$-descending merge link}.

\begin{lemma}[Descending merge link]\label{lem:desc_merge_link}
The $\mu$-descending merge link of $v$ is isomorphic to $VE_{\mu_1(v)}$, hence is $(\nu(\mu_1(v))-1)$-connected.
\end{lemma}

\begin{proof}
A vertex $w$ of $E_m$ with $w<v$ is in the $\mu$-descending merge link of $v$ if and only $\mu_k(w)=\mu_k(v)$ for all $k\ge 3$. This is because if $\mu_k(w)<\mu_k(v)$ for $w<v$ then necessarily $\mu_\ell(w)>\mu_\ell(v)$ for some $\ell>k$ (and on the other hand $\phi(w)<\phi(v)$). Suppose $v=[f^{-1} g]$ and write the domain of $f^{-1} g$ as the disjoint union $L_1\sqcup\cdots\sqcup L_m$ of leaves $L_i$. Without loss of generality (that is, up to the action of $\mathcal{G}(m)$), we can assume that for each $1\le i\le \mu_1(v)$, $L_i$ does not get mapped by $f^{-1} g$ into the same root as any other $L_j$. In other words, the first $\mu_1(v)$ of the leaves are the ones revealing $\mu_1(v)$, and note that this property does not depend on the choice of representative of $v$. Now suppose $w=[f_1^{-1} g_1]<v$ is in the $\mu$-descending merge link of $v$. For a homeomorphism $h$ with domain $\C^S(m)$ write $\res_\alpha(h)$ for the restriction of $h$ to $\C^S_1\sqcup\cdots\sqcup \C^S_{\mu_1(v)}$ and $\res_\omega(h)$ for the restriction of $h$ to $\C^S_{\mu_1(v)+1}\sqcup\cdots\sqcup \C^S_m$. Since $\mu_k(w)=\mu_k(v)$ for all $k\ge 3$ we know that $[\res_\omega(f_1^{-1} g_1)]=[\res_\omega(f^{-1} g)]$. We thus have a well defined injective simplicial map from the $\mu$-descending merge link of $v$ to $VE_{\mu_1(v)}$ given by sending $[f_1^{-1} g_1]$ to $[\res_\alpha(f_1^{-1} g_1)]$. It is straightforward to see that it is also surjective. The connectivity statement follows from Lemma~\ref{lem:VE_conn}.
\end{proof}

\subsection*{The \texorpdfstring{$\boldsymbol{\mu}$}{mu}-descending split link}
Now we focus on the $\mu$-descending split link of $v=[f^{-1}g]$. The $\mu$-descending split link of $v$ lies in the realization of the open interval $(v,[\id_m])$. An element $w$ of $(v,[\id_m])$ is in the $\mu$-descending split link of $v$ if and only $\mu_k(v)>\mu_k(w)$ for some $k\ge 3$. This is because if $\mu_\ell(v)<\mu_\ell(w)$ for $v<w$ then necessarily $\mu_k(v)>\mu_k(w)$ for some $k>\ell$ (and on the other hand $\phi(v)>\phi(w)$ is impossible). Call a root $R_i$ in the range of $f^{-1}g$ \emph{fat} if it has weight at least $3$. Any element of $(v,[\id_m])$ is of the form $[e f^{-1}g]=[(f')^{-1}g]$ for $e$ and $f'$ multicolored forests such that $f=f'e$. Note that the roots $R_i$ comprise the domain of $e$. The above characterization implies that such an element is $\mu$-descending if and only if $e$ is non-trivial on at least one fat root. Consider the subposet
\[
D(f)\defeq\{[e]\in ([\id_q],[f])\mid e\text{ is non-trivial on at least one fat root}\}
\]
of $([\id_q],[f])$. The $\mu$-descending split link of $v=[f^{-1}g]$ is isomorphic to $|D(f)|$ via the map $[(f')^{-1}g]\mapsto [e]$ induced by right multiplication by $g^{-1}f$.

\begin{lemma}[Descending split link, first case]\label{lem:some_small}
If $\mu_2(v)\ne 0$ then $|D(f)|$, and hence the $\mu$-descending split link of $v$, is contractible.
\end{lemma}

\begin{proof}
There is a poset map $\pi$ from $[[\id_q],[f]]$ to itself given by sending $[e]$ to $[\overline{e}]$ where $\bar{e}$ is the homeomorphism coinciding with $e$ on each fat root in the domain and acting as the identity on each non-fat root in the domain (picture $\pi$ as ``pruning'' away any parts of $f$ on non-fat roots). Since $[\overline{e}]\le [e]$ and $[\overline{e}]=[e]$ whenever $[e]$ is in the image of $\pi$, this yields a retraction of $|[[\id_q],[f]]|$ onto the realization of the image of $\pi$. Moreover, if $e$ is non-trivial on at least one fat root then so is $\overline{e}$, so $\pi$ restricts to a retraction of $|D(f)|$ onto its image under $\pi$. Since $\mu_2(v)\ne 0$ we have $\pi([f])\in ([\id_q],[f])$, so $\pi(D(f))$ is conically contractible with cone point $\pi([f])$ (see \cite[Section~1.5]{quillen78}).
\end{proof}

If $\mu_2(v)=0$ then $D(f)=([\id_q],[f])$ and $\pi$ is the identity, so we need to just inspect $([\id_q],[f])$ directly. Say $f=f_1\oplus\cdots\oplus f_q$ for $f_i$ of corank $1$. For each $f_i$, recall that $s\in S$ lies in $\Spec(f_i)$ if there exists $c\in \C^S$ with $f_i(c)(s)\ne c(s)$. Now define the \emph{special spectrum} $\SSpec(f_i)$ to be the set of $s\in\Spec(f_i)$ such that $f_i(c)(s)\ne c(s)$ holds for all $c\in \C^S$ (rather than just some $c$). For example the spectrum of $(x_t\oplus \id)x_s$ is $\{s,t\}$ and the special spectrum is $\{s\}$, whereas the special spectrum of $(x_t\oplus x_t)x_s$ is the whole spectrum $\{s,t\}$.

In the following note that for $v\in E_m$ we have $\mu_1(v)\le m-1$, so $\log_2(m-\mu_1(v))$ is defined (and non-negative).

\begin{lemma}[Descending split link, second case]\label{lem:all_big}
Let $v\in E_m\setminus VE_m$. If $\mu_2(v)=0$ then $|([\id_q],[f])|$, and hence the $\mu$-descending split link of $v$, is $(\log_2(m-\mu_1(v))-3)$-connected.
\end{lemma}

\begin{proof}
Our strategy is to cover $([\id_q],[f])$ by the stars of its minimal elements (i.e., atoms in the lattice $[[\id_q],[f]]$). Each minimal element is of the form $z_i^s\defeq[\id_{i-1}\oplus x_s\oplus \id_{q-i}]$ for some $i$ such that $R_i$ is a fat root and some $s\in \SSpec(f_i)$. Write $Z_i^s$ for the star of $z_i^s$ in $|([\id_q],[f])|$. By construction $|([\id_q],[f])|$ is covered by the collection of those $Z_i^s$ such that $R_i$ a fat root and $s\in \SSpec(f_i)$. Any intersection of complexes $Z_i^s$ equals the star in $|([\id_q],[f])|$ of the join of the corresponding elements $z_i^s$, which is contractible as soon as it is non-empty, and so by the Nerve Lemma~\ref{lem:nerve} it suffices to prove that the nerve of this covering is $(\log_2(m-\mu_1(v))-3)$-connected. If $\SSpec(f_i)\subsetneq\Spec(f_i)$ for some $i$ then the $z_i^s$ all have a common upper bound strictly smaller than $[f]$, and so the nerve is a (contractible) simplex. Now assume $\SSpec(f_i)=\Spec(f_i)$ for all $i$. Then the join of a collection of $z_i^s$ equals $[f]$ if we use every $z_i^s$ and otherwise is a strict lower bound of $[f]$. Hence the nerve of the covering is the boundary of a simplex whose vertices are indexed by pairs $(i,s)$ such that $R_i$ is a fat root and $s\in \Spec(f_i)$. Since $\Spec(f_i)$ is empty if $R_i$ is not fat, the number of vertices of this simplex is $\sum_{i=1}^q |\Spec(f_i)|$, and so the nerve is $((\sum_{i=1}^q |\Spec(f_i)|)-3)$-connected. Each $f_i$ has $2^{|\Spec(f_i)|}$ leaves, and $f$ has $m$ leaves, so $m=\sum_{i=1}^q 2^{|\Spec(f_i)|}$. Since $\log_2(2^a+2^b)\le a+b$ whenever $a,b\ge 1$, we conclude that $\log_2(m-\mu_1(v))\le \sum_{i=1}^q |\Spec(f_i)|$, so we are done.
\end{proof}

\begin{remark}
In the case when $S$ is finite and $G$ is trivial, the proof of Lemma~\ref{lem:all_big} also serves to correct a mistake in the proof of \cite[Lemma~3.8]{fluch13}. There the complex called the ``up-link'' is implied to be a join of some number of non-empty complexes, but in fact it is not quite a join, since we are not allowed to completely split apart every block, so there is one maximal simplex missing. It is easy to show that removing this maximal simplex does not affect the claimed connectivity, and alternatively the nerve argument given here directly proves the claimed connectivity.
\end{remark}

\begin{corollary}\label{cor:mu_dlk}
For any vertex $v\in E_m\setminus VE_m$, the $\mu$-descending link of $v$ is $(\nu(\mu_1(v))+\log_2(m-\mu_1(v))-2)$-connected. In particular its connectivity is at least the minimum of $\nu(m/2)-2$ and $\log_2(m/2)-2$.
\end{corollary}

\begin{proof}
By Lemma~\ref{lem:desc_merge_link} the $\mu$-descending merge link of $v$ is $(\nu(\mu_1(v))-1)$-connected. By Lemmas~\ref{lem:some_small} and~\ref{lem:all_big}, the $\mu$-descending split link of $v$ is $(\log_2(m-\mu_1(v))-3)$-connected, so the first claim follows from the fact that the join of an $a$-connected complex and a $b$-connected complex is $(a+b+2)$-connected. To see the second claim, note that if $\mu_1(v)\ge m/2$ then $\nu(\mu_1(v))+\log_2(m-\mu_1(v))-2 \ge \nu(m/2)-2$ and if $\mu_1(v)\le m/2$ then $\nu(\mu_1(v))+\log_2(m-\mu_1(v))-2 \ge \log_2(m/2)-2$.
\end{proof}

\subsection*{Proof of finiteness theorem}

Now that we have analyzed the connectivity of the \mbox{$\mu$-descending} links we are ready to prove our main finiteness theorem.  This depends on the following proposition.

\begin{proposition}\label{prop:E_hi_conn}
The connectivity of $E_m$ is at least the minimum of $\nu(m/2)-2$ and $\log_2(m/2)-2$.
\end{proposition}

\begin{proof}
By Corollary~\ref{cor:mu_dlk} the connectivity of the $\mu$-descending link of any vertex $v$ in $E_m\setminus VE_m$ is at least the minimum of $\nu(m/2)-2$ and $\log_2(m/2)-2$. By Lemma~\ref{lem:VE_conn} $VE_m$ is $(\nu(m)-1)$-connected, hence $(\nu(m/2)-2)$-connected, so Morse theory (e.g., \cite[Corollary~2.6]{bestvina97}) tells us that the inclusion $VE_m \to E_m$ induces isomorphisms in $\pi_k$ for all $k$ up to the minimum of $\nu(m/2)-2$ and $\log_2(m/2)-2$ and a surjection in $\pi_k$ for $k$ one more than this. Hence the connectivity of $E_m$ is at least the minimum of $\nu(m/2)-2$ and $\log_2(m/2)-2$.
\end{proof}

Since both of these functions go to $\infty$ as $m$ goes to $\infty$, we have shown that the connectivity of $E_m$ goes to $\infty$ as $m$ goes to~$\infty$, i.e., the connectivity of $\dlk [h]$ goes to~$\infty$ as $\phi([h])$ goes to~$\infty$.

\finitenesstheorem*
\begin{proof}
We verify the requirements of Theorem~\ref{thrm:morse_brown}. The complex $X$ is contractible (Proposition~\ref{prop:cible}), the action of $SV_G$ on $X$ has stabilizers of type $\F_n$ (Corollary~\ref{cor:stabs}), and each sublevel set $X_m$ is cocompact (Proposition~\ref{prop:cocpt}). By Proposition~\ref{prop:E_hi_conn} the descending link of a vertex gets arbitrarily highly connected as the $\phi$ value of the vertex goes to infinity. Now Theorem~\ref{thrm:morse_brown} says $SV_G$ is of type $\F_n$.
\end{proof}

\bibliographystyle{alpha}

\begin{thebibliography}{FMWZ13}

\bibitem[Alo94]{alonso94}
Juan~M. Alonso.
\newblock Finiteness conditions on groups and quasi-isometries.
\newblock {\em J. Pure Appl. Algebra}, 95(2):121--129, 1994.

\bibitem[BB97]{bestvina97}
Mladen Bestvina and Noel Brady.
\newblock Morse theory and finiteness properties of groups.
\newblock {\em Invent. Math.}, 129(3):445--470, 1997.

\bibitem[BBM20]{belk20}
James Belk, Collin Bleak, and Francesco Matucci.
\newblock Embedding right-angled {A}rtin groups into {B}rin-{T}hompson groups.
\newblock {\em Math. Proc. Cambridge Philos. Soc.}, 169(2):225--229, 2020.

\bibitem[BdCK15]{bartholdi15}
Laurent Bartholdi, Yves de~Cornulier, and Dessislava~H. Kochloukova.
\newblock Homological finiteness properties of wreath products.
\newblock {\em Q. J. Math.}, 66(2):437--457, 2015.

\bibitem[BEH20]{BEH20}
Collin Bleak, Luke Elliott, and James Hyde.
\newblock Sufficient conditions for a group of homeomorphisms of the {C}antor
  set to be two generated.
\newblock arXiv:2008.04791.

\bibitem[BF19]{belk19}
James Belk and Bradley Forrest.
\newblock Rearrangement groups of fractals.
\newblock {\em Trans. Amer. Math. Soc.}, 372(7):4509--4552, 2019.

\bibitem[BFM{\etalchar{+}}16]{bux16}
Kai-Uwe Bux, Martin~G. Fluch, Marco Marschler, Stefan Witzel, and Matthew C.~B.
  Zaremsky.
\newblock The braided {T}hompson's groups are of type {$\rm F_\infty$}.
\newblock {\em J. Reine Angew. Math.}, 718:59--101, 2016.
\newblock With an appendix by Zaremsky.

\bibitem[BG84]{BG84}
Kenneth~S. Brown and Ross Geoghegan.
\newblock An infinite-dimensional torsion-free $\mathrm{FP}_\infty$ group.
\newblock {\em Inventiones mathematicae}, 77(2):367--381, 1984.

\bibitem[BL10]{bleak10}
Collin Bleak and Daniel Lanoue.
\newblock A family of non-isomorphism results.
\newblock {\em Geom. Dedicata}, 146:21--26, 2010.

\bibitem[BLV{\v{Z}}94]{bjoerner94}
A.~Bj\"orner, L.~Lov\'asz, S.~T. Vre\'cica, and R.~T. {\v{Z}}ivaljevi\'c.
\newblock Chessboard complexes and matching complexes.
\newblock {\em J. London Math. Soc. (2)}, 49(1):25--39, 1994.

\bibitem[BM16]{belk16}
James Belk and Francesco Matucci.
\newblock R\"over's simple group is of type {$\text{F}_\infty$}.
\newblock {\em Publ. Mat.}, 60(2):501--524, 2016.

\bibitem[Bri98]{bridson98}
Martin~R Bridson.
\newblock Controlled embeddings into groups that have no non-trivial finite
  quotients.
\newblock {\em Geometry and Topology Monographs}, 1:99--116, 1998.

\bibitem[Bri04]{brin04}
Matthew~G. Brin.
\newblock Higher dimensional {T}hompson groups.
\newblock {\em Geom. Dedicata}, 108:163--192, 2004.

\bibitem[Bri05]{brin05}
Matthew~G. Brin.
\newblock Presentations of higher dimensional {T}hompson groups.
\newblock {\em J. Algebra}, 284(2):520--558, 2005.

\bibitem[Bri10]{brin10}
Matthew~G. Brin.
\newblock On the baker's map and the simplicity of the higher dimensional
  {T}hompson groups {$nV$}.
\newblock {\em Publ. Mat.}, 54(2):433--439, 2010.

\bibitem[Bro87]{brown87}
Kenneth~S. Brown.
\newblock Finiteness properties of groups.
\newblock In {\em Proceedings of the {N}orthwestern conference on cohomology of
  groups ({E}vanston, {I}ll., 1985)}, volume~44, pages 45--75, 1987.

\bibitem[Bro92]{brown92}
Kenneth~S. Brown.
\newblock The geometry of finitely presented infinite simple groups.
\newblock In {\em Algorithms and classification in combinatorial group theory
  ({B}erkeley, {CA}, 1989)}, volume~23 of {\em Math. Sci. Res. Inst. Publ.},
  pages 121--136. Springer, New York, 1992.

\bibitem[Cam90]{cameron90}
Peter~J. Cameron.
\newblock {\em Oligomorphic permutation groups}, volume 152 of {\em London
  Mathematical Society Lecture Note Series}.
\newblock Cambridge University Press, Cambridge, 1990.

\bibitem[CFP96]{cannon96}
J.~W. Cannon, W.~J. Floyd, and W.~R. Parry.
\newblock Introductory notes on {R}ichard {T}hompson's groups.
\newblock {\em Enseign. Math. (2)}, 42(3-4):215--256, 1996.

\bibitem[CR09]{caprace09}
Pierre-Emmanuel Caprace and Bertrand R\'emy.
\newblock Simplicity and superrigidity of twin building lattices.
\newblock {\em Invent. Math.}, 176(1):169--221, 2009.

\bibitem[CR10]{caprace10}
Pierre-Emmanuel Caprace and Bertrand R\'emy.
\newblock Non-distortion of twin building lattices.
\newblock {\em Geom. Dedicata}, 147:397--408, 2010.

\bibitem[DS20]{darbinyan20}
Arman Darbinyan and Markus Steenbock.
\newblock Embeddings into left-orderable simple groups.
\newblock arXiv:2005.06183.

\bibitem[FMWZ13]{fluch13}
Martin~G. Fluch, Marco Marschler, Stefan Witzel, and Matthew C.~B. Zaremsky.
\newblock The {B}rin--{T}hompson groups {$sV$} are of type {$\text{F}_\infty$}.
\newblock {\em Pacific J. Math.}, 266(2):283--295, 2013.

\bibitem[Gor74]{goryushkin74}
A.~P. Goryushkin.
\newblock Imbedding of countable groups in 2-generated simple groups.
\newblock {\em Mathematical Notes}, 16(2):725--727, 1974.

\bibitem[Hal74]{hall74}
P.~Hall.
\newblock On the embedding of a group in a join of given groups.
\newblock {\em Journal of the Australian Mathematical Society}, 17(4):434--495,
  1974.

\bibitem[HM12]{hennig12}
Johanna Hennig and Francesco Matucci.
\newblock Presentations for the higher-dimensional {T}hompson groups {$nV$}.
\newblock {\em Pacific J. Math.}, 257(1):53--74, 2012.

\bibitem[Hyd17]{hyde17}
James~Thomas Hyde.
\newblock {\em Constructing 2-generated subgroups of the group of
  homeomorphisms of Cantor space}.
\newblock PhD thesis, University of St Andrews, 2017.

\bibitem[KMPN13]{kochloukova13}
Dessislava~H. Kochloukova, Conchita Mart\'inez-P\'erez, and Brita E.~A.
  Nucinkis.
\newblock Cohomological finiteness properties of the {B}rin-{T}hompson-{H}igman
  groups {$2V$} and {$3V$}.
\newblock {\em Proc. Edinb. Math. Soc. (2)}, 56(3):777--804, 2013.

\bibitem[Qui78]{quillen78}
Daniel Quillen.
\newblock Homotopy properties of the poset of nontrivial {$p$}-subgroups of a
  group.
\newblock {\em Adv. in Math.}, 28(2):101--128, 1978.

\bibitem[Sch76]{schupp76}
Paul~E. Schupp.
\newblock Embeddings into simple groups.
\newblock {\em J. London Math. Soc. (2)}, 13(1):90--94, 1976.

\bibitem[Ste92]{stein92}
Melanie Stein.
\newblock Groups of piecewise linear homeomorphisms.
\newblock {\em Trans. Amer. Math. Soc.}, 332(2):477--514, 1992.

\bibitem[SWZ19]{skipper19}
Rachel Skipper, Stefan Witzel, and Matthew C.~B. Zaremsky.
\newblock Simple groups separated by finiteness properties.
\newblock {\em Invent. Math.}, 215(2):713--740, 2019.

\bibitem[Wit19]{witzel19}
Stefan Witzel.
\newblock Classifying spaces from {O}re categories with {G}arside families.
\newblock {\em Algebr. Geom. Topol.}, 19(3):1477--1524, 2019.

\bibitem[WZ18]{witzel18}
Stefan Witzel and Matthew Zaremsky.
\newblock Thompson groups for systems of groups, and their finiteness
  properties.
\newblock {\em Groups Geom. Dyn.}, 12(1):289--358, 2018.

\end{thebibliography}

\newcommand{\etalchar}[1]{$^{#1}$}

\end{document}